\documentclass[11pt]{article}
\usepackage[utf8]{inputenc} 
\usepackage{amsmath}
\usepackage{amsfonts}
\usepackage{amssymb}
\usepackage{ntheorem,lipsum}
\theorembodyfont{\upshape}
\usepackage[T1]{fontenc} 
\newtheorem{thm}{Theorem}
\newtheorem{cor}{Corollary}
\newtheorem{defn}{Definition}
\newtheorem{example}{Example}
\newtheorem{lem}{Lemma}

\newtheorem{prop}{Proposition}

\newenvironment{proof}[1][Proof]{\noindent\textbf{#1.} }{\ \rule{0.5em}{0.5em}}

\begin{document}

\title{\textbf{Uncertainty Principle for the multivariate two sided  continuous quaternion Windowed Fourier transform }}

\author{ Kamel Brahim
\thanks{Department of Mathematics, College of Science, University of Bisha, Bisha 61922, p.o Box 344, Saudi Arabia.}
\thanks{ Faculty of Sciences of Tunis. University of Tunis El Manar, Tunis, Tunisia.\newline  E-mail :
kamel.brahim@fst.utm.tn}
\qquad \& \quad  Emna Tefjeni
\thanks{ Faculty of Sciences of Tunis. University of Tunis El Manar, Tunis, Tunisia.\newline  E-mail :
tefjeni.emna@outlook.fr}}
\date{}
\maketitle
\begin{abstract}
In this paper, we generalize the continuous quaternion windowed Fourier transform on $\mathbb{R}^{2}$ to $\mathbb{R}^{2d}$, called the multivariate two sided continuous quaternion windowed Fourier transform. Using the two sided quaternion Fourier transform, we derive several important properties such as (reconstruction formula, reproducing kernel, plancherel's formula, etc.). We present several example of the multivariate two sided continuous quaternion windowed Fourier transform. We apply the multivariate two sided continuous quaternion windowed Fourier transform properties and the two sided quaternion Fourier transform to establish Lieb uncertainty principle, the Logarithmic uncertainty principle the Beckner's uncertainty principle in term of entropy and the Heisenberg uncertainty principle for the multivariate two sided continuous quaternion windowed Fourier transform, the radar quaternion ambiguity function and the quaternion-Wigner transform. Last we study the multivariate two sided continuous quaternion windowed Fourier transform, the radar quaternion ambiguity function and the quaternion-Wigner transform on subset of finite measures.
\end{abstract}

\section{Introduction}
Uncertainty principles are mathematical results that give limitations on the simultaneous concentration of a function and its Euclidean Fourier transform. They have implications in two main areas: quantum physics and signal analysis. In quantum physics, they tell us that a particle's speed and position cannot both be measured with infinite precision. In signal analysis, they tell us that if we observe a signal only for a finite period of time, we will lose information about the frequencies the signal consists of. There are many ways to get the statement about concentration precise. For more details about uncertainty principles, we refer the reader to \cite{uncertainty1,uncertainty2}.\\
For a quaternion function $f \in L^{2}(\mathbb{R}^{2d},\mathbb{H})$ and a non zero quaternion function $g \in L^{2}(\mathbb{R}^{2d},\mathbb{H})$ called a quaternion window function,
the multivariate two sided  continuous quaternion windowed Fourier transform QWFT of $f$ with respect to $g$ is defined on $\mathbb{R}^{2d}\times\mathbb{R}^{2d}$ by
\begin{equation*}
\mathcal{G}_{g}(f)(x,w) = \displaystyle\int_{\mathbb{R}^{2d}} e^{-is.u} f(y)\overline{g(y-x)} e^{-jt.v}\ d\mu_{2d}(y),\ \ \ w = (u,v), y=(s,t).
\end{equation*}
The multivariate two sided  continuous quaternion Windowed Fourier transform  is closely related to other common and known time frequency distribution as the radar quaternion ambiguity function QAF on $\mathbb{R}^{2d}\times\mathbb{R}^{2d}$ by
\begin{equation*}
A(f,g)(x,w) = \displaystyle\int_{\mathbb{R}^{2d}} e^{-is.u} f(y + \dfrac{x}{2})\overline{g(y-\frac{x}{2})} e^{-jt.v}\ d\mu_{2d}(y),\ \ \ w = (u,v), y=(s,t).
\end{equation*}
and the quaternion Wigner transform QWVT defined on $\mathbb{R}^{2d}\times\mathbb{R}^{2d}$
\begin{equation*}
W(f,g)(x,w) = \displaystyle\int_{\mathbb{R}^{2d}} e^{-is.u} f(x + \dfrac{y}{2})\overline{g(x-\frac{y}{2})} e^{-jt.v}\ d\mu_{2d}(y),\ \ \ w = (u,v), y=(s,t).
\end{equation*}
The aim of this paper is to  generalize the continuous quaternion windowed Fourier transform on $\mathbb{R}^{2}$ to $\mathbb{R}^{2d}$, called the multivariate two sided continuous quaternion windowed Fourier transform which has been started in \cite{article1,article2}.\\
Our purpose in this work  is to prove three Lieb uncertainty principle for both of the QWFT, QWVT and QAF. We also prove the Logarithmic uncertainty principle, the Beckner uncertainty principle in terme of entropy, Heisenberg uncertainty principle for the two sided quaternion windowed Fourier transform. Last we study the multivariate two sided continuous quaternion windowed Fourier transform on subset of finite measures.
Our paper is organized as follows: In section 2, we present basic notions and notations related to the quaternion Fourier transform. In section 3, we recall the definition and some results for the two sided quaternion windowed Fourier transform useful in the sequel. In section 4, we provide  the Lieb uncertainty principle, the Logarithmic uncertainty principle, The Beckner's uncertainty principle in terms of entropy  for the two-sided quaternion windowed Fourier transform.

\section{Generalities}
For convenience of further discussions, we briefly review some basic ideas on quaternions \cite{quaternion} 	and the (two-sided) quaternion Fourier transform \cite{quaternionfourier}. The quaternion algebra over $\mathbb{R}$, denoted by $\mathbb{H}$, is an associative non commutative four-dimensional algebra,
\begin{center}
$\mathbb{H} = \{q = q_{0} + iq_{1} + jq_{2} + kq_{3} \hspace{0.1 cm}|\hspace{0.1 cm} q_{0}, q_{1}, q_{2}, q_{3} \in \mathbb{R}\},$
\end{center}
which obeys Hamilton's multiplication rules
\begin{center}
$ij = - ji = k$,\hspace{0.2 cm} $jk = -kj = i$,\hspace{0.2 cm}  $ki = -ik = j$,\hspace{0.2 cm} $i^{2} = j^{2} = k^{2} = ijk = -1$.
\end{center}
The quaternion conjugate of a quaternion $q$ is given by
\begin{center}
$\overline{q} = q_{0} - iq_{1} - jq_{2} - kq_{3}$, \hspace{0.3 cm} $ q_{0}, q_{1}, q_{2}, q_{3} \in \mathbb{R}$.
\end{center}
The quaternion conjugation is a linear anti-involution
\begin{center}
$\overline{pq} = \overline{q}~\overline{p}$, \hspace{0.3 cm} $\overline{p+q} = \overline{p}+\overline{q}$, \hspace{0.3 cm} $\overline{\overline{p}} = p$.
\end{center}
The modulus of a quaternion $q$ is defined by
\begin{center}
$|q| = \sqrt{q\overline{q}} = \sqrt{q_{0}^{2}+ q_{1}^{2}+q_{2}^{2}+ q_{3}^{2}}$.
\end{center}
It is not difficult to see that
\begin{center}
$|pq| = |p| |q|$, \hspace{0.3 cm} $\forall p,q \in\mathbb{H}$.
\end{center}
The real scalar part $q_{0}$
\begin{center}
$q_{0} = Sc(q)$
\end{center}
leads to a cyclic multiplication symmetry
\begin{center}
$Sc(qrs) = Sc(rsq) \ \ \ q, r,s \in \mathbb{H}$.
\end{center}
Each quaternion can be split by  \cite{quaternionfourier2}
\begin{equation*}
q = q_{+} + q_{-},\ \ \ q_{\pm} = \dfrac{1}{2}(q \pm iqj).
\end{equation*}
By the real components $q_{0},q_{1},q_{2},q_{3}$ in $\mathbb{R}$, we have
\begin{equation*}
q_{\pm} = \{q_{0}\pm q_{3} + i(q_{1}\mp q_{2})\}\dfrac{1\pm k}{2} = \dfrac{1\pm k}{2}\{q_{0}\pm q_{3} + j(q_{1}\pm q_{2})\}.
\end{equation*}
For $q$ in $\mathbb{H}$, we have
\begin{equation*}
|q|^{2} = |q_{+}|^{2} + |q_{-}|^{2} .
\end{equation*}
A quaternion-valued function $f:\mathbb{R}^{d} \longrightarrow \mathbb{H}$ will be written as \begin{center}
$f(x) = f_{0}(x) + i f_{1}(x) + jf_{2}(x) +kf_{3}(x)$,
\end{center}
with real-valued coefficient functions $f_{0}, f_{1}, f_{2}, f_{3}:\mathbb{R}^{d}\longmapsto\mathbb{R}$.\\
We defined by $\mu_{d}$ the normalized Lebesgue measure on  $\mathbb{R}^{d}$ by $d\mu_{d}(x) = \dfrac{dx}{(2\pi)^{\frac{d}{2}}}.$\\
If $1\leq p < \infty$, the $L^{p}$-norm of $f: \mathbb{R}^{d} \longrightarrow \mathbb{H}$ is defined by
\begin{equation}\label{normp}
\|f\|_{p,d} = \bigg(\displaystyle\int_{\mathbb{R}^{d}} |f(x)|^{p} d\mu_{d}(x)\bigg)^{\frac{1}{p}}.
\end{equation}
For $p = \infty$, $L^{\infty}(\mathbb{R}^{d},\mathbb{H})$ is a collection of essentially bounded measurable functions with the norm
$$\|f\|_{\infty} = ess \displaystyle\sup_{x\in \mathbb{R}^{d}} |f(x)|$$
if $f \in L^{\infty}(\mathbb{R}^{d},\mathbb{H})$ is continuous then
\begin{equation}\label{norminfty}
\|f\|_{\infty} = \displaystyle\sup_{x\in \mathbb{R}^{d}} |f(x)|.
\end{equation}
For $p=2$, we can define the quaternion-valued inner product
\begin{equation}\label{&}
(f,g)_{2,d} = \displaystyle\int_{\mathbb{R}^{d}} f(x) \overline{g(x)} d\mu_{d}(x)
\end{equation}
with symmetric real scalar part
\begin{eqnarray}\label{pro}
\langle f,g\rangle_{2,d}\ & = & \ \dfrac{1}{2}\big[(f,g)_{2,d} + (g,f)_{2,d}\big]\ = \ \displaystyle\int_{\mathbb{R}^{d}} Sc(f(x) \overline{g(x)}) d\mu_{d}(x)\nonumber\\
& = & Sc\bigg(\displaystyle\int_{\mathbb{R}^{d}} f(x) \overline{g(x)} d\mu_{d}(x)\bigg).
\end{eqnarray}
Both (\ref{&}) and (\ref{pro}) lead to the $L^{2}(\mathbb{R}^{d},\mathbb{H})$-norm
\begin{equation}
\|f\|_{2,d} = \sqrt{(f,f)_{2,d}} = \sqrt{\langle f,f\rangle_{2,d}}  =~\left(\int_{\mathbb{R}^{d}} |f(x)|^{2} d\mu_{d}(x)\right)^{\frac{1}{2}}.
\end{equation}
As a consequence of the inner product (\ref{pro}) we obtain the quaternion Cauchy-Schwartz inequality
\begin{eqnarray}\label{cs}
\bigg|\displaystyle\int_{\mathbb{R}^{d}} f(x)\overline{g(x)} d\mu_{d}(x)\bigg| \leq \bigg(\displaystyle\int_{\mathbb{R}^{d}} |f(x)|^{2}d\mu_{d}(x)\bigg)^{\frac{1}{2}}\bigg(\displaystyle\int_{\mathbb{R}^{d}} |g(x)|^{2}d\mu_{d}(x)\bigg)^{\frac{1}{2}}\hspace{0.2 cm} \forall f , g \in L^{2}(\mathbb{R}^{d},\mathbb{H}).
\end{eqnarray}
For two function $f$, $g$ $\in L^{2}(\mathbb{R}^{d},\mathbb{H})$. Using (\ref{pro}) and (\ref{cs}) the Schwartz inequality takes the form
\begin{equation}\label{hes1}
\bigg(\displaystyle\int_{\mathbb{R}^{d}}(g(x)\overline{f(x)} + f(x)\overline{g(x)})d\mu_{d}(x)\bigg)^{2} \leq 4\ \bigg(\displaystyle\int_{\mathbb{R}^{d}}|f(x)|^{2}d\mu_{d}(x)\bigg) \bigg(\displaystyle\int_{\mathbb{R}^{d}}|g(x)|^{2}d\mu_{d}(x)\bigg).
\end{equation}
The convolution of $f \in L^{2}(\mathbb{R}^{d},\mathbb{H})$ and $g \in L^{2}(\mathbb{R}^{d},\mathbb{H})$, denoted by $f*g$, is defined by
\begin{equation}\label{conv}
(f*g)(x) = \displaystyle\int_{\mathbb{R}^{d}} f(y) g(x - y) d\mu_{d}(y).
\end{equation}
For every $x$ $\in \mathbb{R}^{d}$, we denote by $T_{x}$ the translation operator defined by
\begin{center}
$\forall t \in \mathbb{R}^{d}$, \hspace{0.3 cm} $T_{x}f(t) = f(t-x)$.
\end{center}
Let $\lambda > 0$, we denote by $f_{\lambda}$ the dilate of $f$ by $f_{\lambda}(x) = f(\lambda x)$.
\begin{defn}\label{definition1}
\mbox{}\\
Let $\alpha = (\alpha_{1},\alpha_{2},\alpha_{2},\dots,\alpha_{d})$ be a multi-index of non negative integers. One denote
\begin{center}
$|\alpha| = \displaystyle\sum_{p=1}^{d}\alpha_{p}$\hspace{0.5 cm} and \hspace{0.5 cm}$\alpha! = \displaystyle\prod_{p=1}^{d} \alpha_{p}$.
\end{center}
and for $x = (x_{1},x_{2},\alpha_{2},\dots,x_{d}) \in \mathbb{R}^{d}$, \hspace{0.3 cm}$x^{\alpha} = \displaystyle\prod_{p=1}^{d} x_{p}^{\alpha_{p}}$.\\
Derivative are conveniently expressed by multi-indices
\begin{center}
$d^{\alpha} = \dfrac{d^{|\alpha|}}{dx_{1}^{\alpha_{1}}dx_{2}^{\alpha_{2}}dx_{3}^{\alpha_{3}}......dx_{d}^{\alpha_{d}}}.$
\end{center}
Next we obtain the Schwartz space
\begin{center}
$\mathcal{S}(\mathbb{R}^{d},\mathbb{H}) = \bigg\{f \in C^{\infty}(\mathbb{R}^{d},\mathbb{H}) ; \displaystyle\sup_{x\in\mathbb{R}^{d}} (1 + |x|^{p})|d^{\alpha}f(x)|< \infty\bigg\}$
\end{center}
where $C^{\infty}(\mathbb{R}^{d},\mathbb{H})$ is the set of smooth function from $\mathbb{R}^{d}$ to $\mathbb{H}$.
\end{defn}
The quaternion Fourier transform is defined similarly to the classical Fourier transform of the 2D functions. The non commutative property of quaternion multiplication allows us to have three different definitions of the quaternion Fourier transform QFT. In the following we briefly introduce the two sided QFT. For more details we refer the reader to \cite{fields,8,sangwine}.
\begin{defn}
The two sided Quaternionic Fourier transform of a function $f \in L^{1}(\mathbb{R}^{2d},\mathbb{H})\cap L^{2}(\mathbb{R}^{2d},\mathbb{H})$ is defined by \cite{quaternionfourier1,7}
\begin{equation}\label{QFT}
\mathcal{F}_{Q}(f)(u,v) = \displaystyle\int_{\mathbb{R}^{2d}} e^{-ix.u} f(x,y)\ e^{-jy.v} d\mu_{2d}(x,y).
\end{equation}
It satisfies Plancherel's formula
\begin{equation}\label{plan2}
\|f\|_{2,2d} = \|\mathcal{F}_{Q}(f)\|_{2,2d}.
\end{equation}
As a consequence $\mathcal{F}_{Q}$ extends to a unitary operator on $L^{2}(\mathbb{R}^{2d},\mathbb{H})$ and satisfies Parseval's formula
\begin{equation}\label{plan1}
\langle f,g\rangle_{2,2d} = \langle\mathcal{F}_{Q}(f),\mathcal{F}_{Q}(g)\rangle_{2,2d}.
\end{equation}
The inverse quaternion Fourier transform of a function $f\in L^{1}(\mathbb{R}^{2d},\mathbb{H})$ with $\mathcal{F}_{Q}(f)\in L^{1}(\mathbb{R}^{2d},\mathbb{H})$ is given as
\begin{equation*}
f(x,y)  = \displaystyle\int_{\mathbb{R}^{2d}} e^{ix.u} \mathcal{F}_{Q}(f)(u,v)e^{jy.v} d\mu_{2d}(u,v).
\end{equation*}
By the two dimensional plane split of the quaternion signal, the quaternion Fourier transform becomes
\begin{equation*}
\mathcal{F}_{Q}(f)(u,v) = \mathcal{F}_{Q}(f_{+}+f_{-})(u,v) = \mathcal{F}_{Q}(f_{+})(u,v) + \mathcal{F}_{Q}(f_{-})(u,v).
\end{equation*}
The quaternion Fourier transform of $f_{\pm}$ have simple complex forms \cite{quaternionfourier2,quaternionfourier3}
\begin{equation*}
\mathcal{F}_{Q}(f_{\pm})(u,v) = \displaystyle\int_{\mathbb{R}^{2d}} e^{-i(x.u \mp y.v)} f_{\pm}(x,y)\ d\mu_{2d}(x,y),
\end{equation*}
then
\begin{equation*}
\mathcal{F}_{Q}(f_{\pm})(u,v) = \mathcal{F}_{c}(f_{+})(u,-v) + \mathcal{F}_{c}(f_{-})(u,v),
\end{equation*}
where $\mathcal{F}_{c}$ is a complex Fourier transform defined by
\begin{equation*}
\mathcal{F}_{c}(f)(\xi) = \displaystyle\int_{\mathbb{R}^{2d}} e^{-ix.\xi} f(x) d\mu_{2d}(x).
\end{equation*}
Due to $|q|^{2} = |q_{-}|^{2} + |q_{+}|^{2}$, for $f:\mathbb{R}^{2d}\longrightarrow \mathbb{H}$, we have the following two identities \cite{quaternionfourier2,quaternionfourier3}
\begin{equation*}
|f(x)|^{2} = |f_{+}(x)|^{2} + |f_{-}(x)|^{2},\ \mbox{and} \ |\mathcal{F}_{Q}(f)(\xi)|^{2} = |\mathcal{F}_{Q}(f_{+})(\xi)|^{2} + |\mathcal{F}_{Q}(f_{-})(\xi)|^{2}.
\end{equation*}
Let $f$ $\in L^{1}(\mathbb{R}^{2d},\mathbb{H})$; $\alpha$, $\beta > 0$ and $(x,y) , (w,\sigma) \in \mathbb{R}^{2d}$ then
\begin{center}
$\mathcal{F}_{Q}(f(\frac{x}{\alpha},\frac{y}{\beta}))(w,\sigma) = \alpha^{d} \beta^{d} \mathcal{F}_{Q}(f(x,y))(\alpha w , \beta \sigma) .$
\end{center}
Consider a two-dimensional Gaussian function of the form $f(x) = e^{-(\frac{a|x|^{2}+b|y|^{2}}{2})}$ where $a, b$ are non-zero,  positive constants. Then the QFT of $f$ is given by
\begin{center}
$\mathcal{F}_{Q}(e^{-\frac{a |x|^{2}+b |y|^{2}}{2}})(w,\sigma) =\dfrac{1}{(ab)^{\frac{d}{2}}} e^{-(\frac{|w|^{2}}{2a}+\frac{|\sigma|^{2}}{2b})}$.
\end{center}
\end{defn}
\begin{lem}\label{lemma1}(Derivative theorem)\cite{7}
\begin{enumerate}
\item[(1)]Let $f$ in $L^{1}(\mathbb{R}^{2d},\mathbb{H})\cap L^{2}(\mathbb{R}^{2d},\mathbb{H})$. Let $\frac{d^{\gamma}}{dx^{\gamma}}f$ $\in L^{2}(\mathbb{R}^{2d},\mathbb{H})$, $\alpha,\gamma$ in $ \mathbb{N}^{d}$; $|\gamma|\leq |\alpha|$, then we have
\begin{equation*}
\mathcal{F}_{Q}\bigg(\dfrac{d^{\alpha}}{dx^{\alpha}}f(x,y)\bigg)(u,v) = i^{|\alpha|} u^{\alpha} \mathcal{F}_{Q}(f)(u,v).
\end{equation*}
\item[(2)]Let $f$ in $L^{1}(\mathbb{R}^{2d},\mathbb{H})\cap L^{2}(\mathbb{R}^{2d},\mathbb{H})$. Let $\frac{d^{\xi}}{dy^{\xi}}f$ $\in L^{2}(\mathbb{R}^{2d},\mathbb{H})$, $\beta,\xi$ in $ \mathbb{N}^{d}$; $|\xi|\leq |\beta|$, then we have
\begin{equation*}
\mathcal{F}_{Q}\bigg(\dfrac{d^{\beta}}{ dy^{\beta}}f(x,y)\bigg)(u,v) = \mathcal{F}_{Q}(f)(u,v) j^{|\beta|} v^{\beta}.
\end{equation*}
\end{enumerate}
\end{lem}
\begin{lem}\label{lemma2}
\mbox{}\\
Let $f$ in $L^{1}(\mathbb{R}^{2d},\mathbb{H})\cap L^{2}(\mathbb{R}^{2d},\mathbb{H})$. If $\frac{d}{dx_{p}}f$ exist and are in $L^{2}(\mathbb{R}^{2d},\mathbb{H})$ for $p \in \{1,...2d\}$, then
\begin{equation}\label{hes2}
\displaystyle\int_{\mathbb{R}^{2d}} \bigg|\dfrac{d}{dx_{p}}f(x)\bigg|^{2}d\mu_{2d}(x) = \displaystyle\displaystyle\int_{\mathbb{R}^{2d}} w_{p}^{2}|\mathcal{F}_{Q}(f)(w)|^{2} d\mu_{2d}(w).
\end{equation}
\end{lem}

As consequence of  (\ref{hes2}) we immediately obtain the following theorem.
\begin{thm}\label{theorem1} (Component-wise uncertainty principle for $\mathcal{F}_{Q}$)
\mbox{}\\
Let  $f$ in $L^{1}(\mathbb{R}^{2d},\mathbb{H})\cap L^{2}(\mathbb{R}^{2d},\mathbb{H})$. we have the following inequality
\begin{equation}\label{H&sh}
\bigg(\displaystyle\int_{\mathbb{R}^{2d}} x_{p}^{2} |f(x)|^{2} d\mu_{2d}(x)\bigg) \bigg(\displaystyle\int_{\mathbb{R}^{2d}} w_{p}^{2} |\mathcal{F}_{Q}(f)(w)|^{2} d\mu_{2d}(w)\bigg)\geq \dfrac{1}{4} \bigg(\displaystyle\int_{\mathbb{R}^{2d}} |f(x)|^{2} d\mu_{2d}(x) \bigg)^{2}
\end{equation}
such as $p \in \{1,2,.....,2d\}$.
\end{thm}
\begin{proof}
First, let $p \in \{1,2,.....,2d\}$, this inequality is true if
\begin{equation*}
\displaystyle\int_{\mathbb{R}^{2d}} x_{p}^{2} |f(x)|^{2} d\mu_{2d}(x)=+\infty\ \ \mbox{or} \ \
\displaystyle\int_{\mathbb{R}^{2d}} w_{p}^{2} |\mathcal{F}_{Q}(f)(w)|^{2} d\mu_{2d}(w)=+\infty.
\end{equation*}
 We assume in the sequel that $\|x_{p}f\|_{2,2d} < +\infty$ and  $\|w_{p}\mathcal{F}_{Q}(f)\|_{2,2d}< +\infty$. Applying lemma \ref{lemma2}, we  obtain \\ \ \\
$\bigg(\displaystyle\int_{\mathbb{R}^{2d}} x_{p}^{2} |f(x)|^{2} d\mu_{2d}(x)\bigg) \bigg(\displaystyle\int_{\mathbb{R}^{2d}} w_{p}^{2} |\mathcal{F}_{Q}(f)(w)|^{2} d\mu_{2d}(w)\bigg)$
\begin{eqnarray*}
 & = & \bigg(\displaystyle\int_{\mathbb{R}^{2d}} x_{p}^{2} |f(x)|^{2} d\mu_{2d}(x)\bigg) \bigg(\displaystyle\int_{\mathbb{R}^{2d}}  \bigg|\dfrac{d}{dx_{p}}f(x)\bigg|^{2} d\mu_{2d}(w)\bigg)\\
& \geq &  \dfrac{1}{4}\bigg(\displaystyle\int_{\mathbb{R}^{2d}} x_{p}\bigg( f(x) \overline{\dfrac{d}{dx_{p}}f(x)} + \overline{f(x)} \dfrac{d}{dx_{p}}f(x)\bigg) d\mu_{2d}(w)\bigg)^{2}\\
&  =  & \dfrac{1}{4}\bigg(\displaystyle\int_{\mathbb{R}^{2d}} x_{p}\dfrac{d}{dx_{p}}\bigg( f(x) \overline{f(x)}\bigg) d\mu_{2d}(w)\bigg)^{2}.
\end{eqnarray*}
Second, using integration by parts we further get
\begin{center}
$\bigg(\displaystyle\int_{\mathbb{R}^{2d}} x_{p}^{2} |f(x)|^{2} d\mu_{2d}(x)\bigg) \bigg(\displaystyle\int_{\mathbb{R}^{2d}} w_{p}^{2} |\mathcal{F}_{Q}(f)(w)|^{2} d\mu_{2d}(w)\bigg) \geq \dfrac{1}{4}\bigg(\displaystyle\int_{\mathbb{R}^{2d}} |f(x)|^{2} d\mu_{2d}(x)\bigg)^{2}$.
\end{center}
\end{proof}

Using $|x|^{2} = \displaystyle\sum_{p=1}^{2d}x_{p}^{2} \geq x_{p}^{2}$, we get the following corollary.

\begin{cor}
\mbox{}\\
For every function $f$ in $L^{1}(\mathbb{R}^{2d},\mathbb{H})\cap L^{2}(\mathbb{R}^{2d},\mathbb{H})$, we have
\begin{equation}\label{new3}
 \left(\int_{\mathbb{R}^{2d}}|x|^{2}|f(x)|^2 d\mu_{2d}(x)\right)\left(\int_{\mathbb{R}^{2d}}|w|^{2}|\mathcal{F}_{Q}(f)(w)|^2 d\mu_{2d}(w) \right)\geq\frac{1}{4}\left(\int_{\mathbb{R}^{2d}}|f(x)|^2 d\mu_{2d}(x)\right)^2.
\end{equation}
\end{cor}
In \cite{beckner} Beckner used Stein-Weiss and Pitt's inequalities to obtain a logarithmic estimate of the uncertainty, he showed that for every $f$ in $\mathcal{S}(\mathbb{R}^{d},\mathbb{C})$ we have
\begin{equation}
\displaystyle\int_{\mathbb{R}^{d}} ln|x| |f(x)|^{2}d\mu_{d}(x) + \displaystyle\int_{\mathbb{R}^{d}} ln|w| |\mathcal{F}_{c}(f)(w)|^{2}d\mu_{d}(w) \geq D_{d} \displaystyle\int_{\mathbb{R}^{d}}|f(x)|^{2}d\mu_{d}(x),
\end{equation}
where
\begin{equation}\label{Dd}
D_{d} = \bigg(\dfrac{\Gamma'(\frac{d}{4})}{\Gamma(\frac{d}{4})} + ln(2) \bigg).
\end{equation}
\begin{thm}\label{logarithm}(Logarithmic uncertainty principle for $\mathcal{F}_{Q}$) For $f$ in $\mathcal{S}(\mathbb{R}^{2d},\mathbb{H})$, we have
\begin{equation}\label{loga}
\displaystyle\int_{\mathbb{R}^{2d}} ln|x| |f(x)|^{2}d\mu_{2d}(x) + \displaystyle\int_{\mathbb{R}^{2d}} ln|w| |\mathcal{F}_{Q}(f)(w)|^{2}d\mu_{2d}(w) \geq D_{2d} \displaystyle\int_{\mathbb{R}^{2d}}|f(x)|^{2}d\mu_{2d}(x),
\end{equation}
where $D_{2d}$ is given by (\ref{Dd}).
\end{thm}
\begin{proof}
We have the following equality,
\begin{eqnarray*}
\displaystyle\int_{\mathbb{R}^{2d}} ln|(u,v)| |\mathcal{F}_{c}(f_{+})(u,-v)|^{2} d\mu_{2d}(u,v)
& = &
\displaystyle\int_{\mathbb{R}^{2d}} ln|\sqrt{|u|^{2}+|-v|^{2}}| |\mathcal{F}_{c}(f_{+})(u,v)|^{2} d\mu_{2d}(u,v)\\
& = & \displaystyle\int_{\mathbb{R}^{2d}} ln|\sqrt{|u|^{2}+|v|^{2}}| |\mathcal{F}_{c}(f_{+})(u,v)|^{2} d\mu_{2d}(u,v)\\
& = & \displaystyle\int_{\mathbb{R}^{2d}} ln|(u,v)| |\mathcal{F}_{c}(f_{+})(u,v)|^{2} d\mu_{2d}(u,v),
\end{eqnarray*}
then
\begin{equation*}
\displaystyle\int_{\mathbb{R}^{2d}} |\mathcal{F}_{Q}(f)(u,v)|^{2} d\mu_{2d}(u,v)  =  \displaystyle\int_{\mathbb{R}^{2d}} |\mathcal{F}_{Q}(f_{+})(u,v)|^{2} d\mu_{2d}(u,v) +  \displaystyle\int_{\mathbb{R}^{2d}} |\mathcal{F}_{Q}(f_{-})(u,v)|^{2} d\mu_{2d}(u,v)
\end{equation*}
Using the Logarithmic uncertainty principle for $f_{\pm}$ together and by the modulus identities, we get\\
$\displaystyle\int_{\mathbb{R}^{2d}} ln|x| |f(x)|^{2} d\mu_{2d}(x)
+
\displaystyle\int_{\mathbb{R}^{2d}} ln|w| |\mathcal{F}_{Q}(f)(w)|^{2} d\mu_{2d}(w)$
\begin{eqnarray*}
& = & \displaystyle\int_{\mathbb{R}^{2d}} ln|x| |f_{+}(x)|^{2} d\mu_{2d}(x) + \displaystyle\int_{\mathbb{R}^{2d}} ln|x| |f_{-}(x)|^{2} d\mu_{2d}(x)\\
& + & \displaystyle\int_{\mathbb{R}^{2d}} ln|w| |\mathcal{F}_{Q}(f_{+})(w)|^{2} d\mu_{2d}(w) + \displaystyle\int_{\mathbb{R}^{2d}} ln|w| |\mathcal{F}_{Q}(f_{-})(w)|^{2} d\mu_{2d}(w)\\
& \geq & D_{2d}\bigg(\displaystyle\int_{\mathbb{R}^{2d}}|f_{+}(x)|^{2}d\mu_{2d}(x) + \displaystyle\int_{\mathbb{R}^{2d}}|f_{-}(x)|^{2}d\mu_{2d}(x)\bigg)\\
& = & D_{2d} \displaystyle\int_{\mathbb{R}^{2d}}|f(x)|^{2}d\mu_{2d}(x).
\end{eqnarray*}
\end{proof}
\section{Two sided Quaternionic Windowed Fourier transform (QWFT)}
 In this section, we present the multivariate continuous two sided quaternion windowed Fourier transform. We investigate several basic properties of the QWFT which are important for signal representation in signal processing. For more details on quaternion windowed Fourier transform, the reader can see \cite{article1, article2, 6, 15, windowed1}.
\begin{defn}
Let $g \in L^{2}(\mathbb{R}^{2d},\mathbb{H})\setminus{\{0\}}$, we denote by $\mathcal{G}_{g}$, the QWFT on $L^{2}(\mathbb{R}^{2d},\mathbb{H})$. The QWFT of $f \in L^{2}(\mathbb{R}^{2d},\mathbb{H})$ with respect to $g$ is defined by
\begin{equation}\label{qw}
\mathcal{G}_{g}(f)(x,w) = \displaystyle\int_{\mathbb{R}^{2d}} e^{-is.u} f(y)\overline{g(y-x)} e^{-jt.v} d\mu_{2d}(y),\ \ \ w=(u,v), y=(s,t).
\end{equation}
We say that $g$ is a (non-zero) quaternion window function.
\end{defn}
\begin{prop}
\mbox{}\\
Let $g$ be a quaternion windowed function. For every $f \in L^{2}(\mathbb{R}^{2d},\mathbb{H})$ and for $(x,w) \in \mathbb{R}^{2d}\times\mathbb{R}^{2d}$, we have
\begin{equation}\label{wq}
\mathcal{G}_{g}(f)(x,w) = \mathcal{F}_{Q}(f\overline{T_{x}g})(w).
\end{equation}
\end{prop}
\begin{prop} For every $f \in L^{2}(\mathbb{R}^{2d},\mathbb{H}),$ we denote by $f_{\lambda}$ the function defined by $f_{\lambda}(t) = f(\lambda t), \lambda>0$. We have that $f_{\lambda}$ belongs to  $L^{2}(\mathbb{R}^{2d},\mathbb{H})$ and $\mathcal{G}_{g}(f_{\lambda})$ is given by
\begin{equation}\label{dilate}
\mathcal{G}_{g}(f_{\lambda})(x,w) = \dfrac{1}{\lambda^{2d}} \mathcal{G}_{g_{\frac{1}{\lambda}}}(f)(\lambda x, \frac{w}{\lambda}).
\end{equation}
\end{prop}
\begin{proof} We have
\begin{eqnarray*}
\mathcal{G}_{g}(f_{\lambda})(x,w) & = & \displaystyle\int_{\mathbb{R}^{2d}}e^{-is.u} f_{\lambda}(y) g(y-x) e^{-jt.v} d\mu_{2d}(y)\\
 & = & \displaystyle\int_{\mathbb{R}^{2d}}e^{-is.u} f(\lambda y) g(y-x) e^{-jt.v} d\mu_{2d}(y)\\
  & = & \dfrac{1}{\lambda^{2d}}\displaystyle\int_{\mathbb{R}^{2d}}e^{-i u.\frac{k}{\lambda}} f(z) g\bigg( \frac{z-\lambda x}{\lambda}\bigg) e^{-j v.\frac{l}{\lambda}} d\mu_{2d}(z)\\
  & = & \dfrac{1}{\lambda^{2d}} \mathcal{G}_{g_{\frac{1}{\lambda}}}(f)(\lambda x, \frac{w}{\lambda}),
\end{eqnarray*}
where $\lambda x = (\lambda x_{1},....,\lambda x_{2d})$.
\end{proof}
\begin{prop} Let $f$, $g$ $\in L^{2}(\mathbb{R}^{2d},\mathbb{H})$. The Quaternion window Fourier transform $\mathcal{G}_{g}(f)$ is uniformly continuous and bounded on the time-frequency plane $\mathbb{R}^{2d}\times\mathbb{R}^{2d}$ and satisfies
\begin{equation}\label{1}
\|\mathcal{G}_{g}(f)\|_{\infty,4d} \leq \|f\|_{2,2d}\|g\|_{2,2d} .
\end{equation}
\end{prop}
\begin{proof}
\begin{enumerate}
\item
We see that any function $f$ from $\mathbb{R}^{2d}$ to $\mathbb{H}$ can be expressed as $f(x) = f_{1}(x) + j f_{2}(x)$ where $f_{1}$, $f_{2}$ are both complex valued functions.
 Let $f(x) = f_{1}(x) + j f_{2}(x)$, $g(x) = g_{1}(x) + j g_{2}(x)$ then
\begin{equation*}
f(x)\overline{g(x)} = (f_{1}(x) + j f_{2}(x))(\overline{g_{1}(x)} - \overline{g_{2}(x)}j) =
(f_{1}(x)\overline{g_{1}(x)} + \overline{f_{2}(x)}g_{2}(x)) + (\overline{f_{2}(x)}g_{1}(x) - f_{1}(x)\overline{g_{2}(x)})j
\end{equation*} then we have
\begin{eqnarray*}
\mathcal{G}_{g}(f)(x,w) = (\mathcal{G}_{g_{1}}(f_{1})(x,w) + \mathcal{G}_{\overline{g_{2}}}(\overline{f_{2}})(x,w)) +
(\mathcal{G}_{\overline{g_{1}}}(\overline{f_{2}})(x,w) - \mathcal{G}_{g_{2}}(f_{1})(x,w) )j.
\end{eqnarray*}
On the other hand, from \cite[eq-21]{fields} we have
\begin{equation*}
e^{-is.u}e^{-jt.v} = e^{-i(s.u+t.v)}\dfrac{1-k}{2} + e^{-i(s.u-t.v)}\dfrac{1+k}{2};
\end{equation*}
then
\begin{eqnarray*}
\mathcal{G}_{g}(f)(x,w) & = & (\mathcal{G}^{c}_{g_{1}}(f_{1})(x,w) + \mathcal{G}^{c}_{\overline{g_{2}}}(\overline{f_{2}})(x,w))\dfrac{1-k}{2} +
(\mathcal{G}^{c}_{\overline{g_{1}}}(\overline{f_{2}})(x,w) - \mathcal{G}^{c}_{g_{2}}(f_{1})(x,w) )\dfrac{1-k}{2}j\\
& + & (\mathcal{G}^{c}_{g_{1}}(f_{1})(x,(u,-v)) + \mathcal{G}^{c}_{\overline{g_{2}}}(\overline{f_{2}})(x,(u,-v)))\dfrac{1+k}{2}\\
& + &
(\mathcal{G}^{c}_{\overline{g_{1}}}(\overline{f_{2}})(x,(u,-v)) - \mathcal{G}^{c}_{g_{2}}(f_{1})(x,(u,-v)) )\dfrac{1+k}{2}j,
\end{eqnarray*}
where $\mathcal{G}^{c}_{g}(f)$ is a complex window function defined by
\begin{equation*}
\mathcal{G}^{c}_{g}(f)(x,w) = \displaystyle\int_{\mathbb{R}^{2d}} f(y) \overline{g(y-x)}e^{-iw.y} d\mu_{2d}(y).
\end{equation*}
From \cite[p.39-lemma 3.1.1]{grochening} we have $\mathcal{G}^{c}_{g}(f)$ is uniformly continuous on $\mathbb{R}^{2d}\times\mathbb{R}^{2d}$ , and we now that the sum of a finite number of uniformly continuous functions is a uniformly continuous function, we deduce that $\mathcal{G}_{g}(f)$ is uniformly continuous.
\item
Let $f$ , $g$ $\in L^{2}(\mathbb{R}^{2d},\mathbb{H})$ and $g \neq 0$ . According to Cauchy Schwartz inequality and for every $(x,w) \in \mathbb{R}^{2d}\times\mathbb{R}^{2d}$
\begin{eqnarray*}
|\mathcal{G}_{g}(f)(x,w)| & = & \bigg|\displaystyle\int_{\mathbb{R}^{2d}} e^{-is.u} f(y) \overline{g(y-x)} e^{-jt.v} d\mu_{2d}(y)\bigg| \leq \displaystyle\int_{\mathbb{R}^{2d}} |f(y)| |g(y-x)| d\mu_{2d}(y)\\
& \leq & \|f\|_{2,2d} \|g\|_{2,2d}
\end{eqnarray*}
consequently
\begin{equation*}
\|\mathcal{G}_{g}(f)\|_{\infty,4d} \leq \|f\|_{2,2d} \|g\|_{2,2d}.
\end{equation*}
\end{enumerate}
\end{proof}

\begin{thm}(Lieb inequality)
\mbox{}\\
Let $g$ be a non zero quaternion window function. For every $f \in L^{2}(\mathbb{R}^{2d},\mathbb{H})$ and $p \geq 2$ we have
\begin{equation}
\|\mathcal{G}_{g}(f)\|_{p,4d} \leq C_{p,q} \|f\|_{2,2d} \|g\|_{2,2d} ,\hspace{2 cm} \frac{1}{p} + \frac{1}{q} = 1
\end{equation}
when $C_{p,q} = \bigg(\dfrac{4}{p}\bigg)^{\frac{d}{p}}  \bigg(\dfrac{1}{q}\bigg)^{\frac{d}{q}}$.
\end{thm}
\begin{proof}
According to Cauchy-Schwartz inequality, $f\overline{T_{x}g} \in L^{1}(\mathbb{R}^{2d},\mathbb{H})$. By Plancherel formula we get $\mathcal{G}_{g}(f) \in L^{2}(\mathbb{R}^{2d}\times\mathbb{R}^{2d},\mathbb{H})$ . In particular, for almost every $x \in \mathbb{R}^{2d}$, $\mathcal{G}_{g}(f)(x,.) = \mathcal{F}_{Q}\{f\overline{T_{x}g}\} \in L^{2}(\mathbb{R}^{2d},\mathbb{H})$ and consequently we deduce that $f\overline{T_{x}g} \in L^{1}(\mathbb{R}^{2d},\mathbb{H})\cap L^{2}(\mathbb{R}^{2d},\mathbb{H})$. This implies that $f\overline{T_{x}g} \in L^{q}(\mathbb{R}^{2d},\mathbb{H})$; where $\frac{1}{p} + \frac{1}{q} = 1$ and $p \geq 2$.\\
Using the Hausdorff-Young theorem, we get
\begin{eqnarray*}
\bigg(\displaystyle\int_{\mathbb{R}^{2d}} |\mathcal{G}_{g}(f)(x,w)|^{p} d\mu_{2d}(w)\bigg)^{\frac{1}{p}} & = & \bigg(\displaystyle\int_{\mathbb{R}^{2d}} |\mathcal{F}_{Q}\{f\overline{T_{x}g})\}(w)|^{p} d\mu_{2d}(w)\bigg)^{\frac{1}{p}}\\
 & \leq & \bigg(\displaystyle\int_{\mathbb{R}^{2d}} |(f\overline{T_{x}g}(y)|^{q} d\mu_{2d}(y)\bigg)^{\frac{1}{q}}\\
& = & \bigg(\displaystyle\int_{\mathbb{R}^{2d}} |f(y)|^{q}|g(y-x)|^{q} d\mu_{2d}(y)\bigg)^{\frac{1}{q}}  = (|f|^{q}*|\check{g}|^{q}(x))^{\frac{1}{q}},
\end{eqnarray*}
where $\check{g}(x) = g(-x)$ then,
\begin{eqnarray*}
\|\mathcal{G}_{g}(f)\|_{p,4d} & = & \bigg(\displaystyle\int_{\mathbb{R}^{2d}}\bigg(\displaystyle\int_{\mathbb{R}^{2d}} |\mathcal{G}_{g}(f)(x,w)|^{p} d\mu_{2d}(x)\bigg)d\mu_{2d}(w)\bigg)^{\frac{1}{p}} \leq \bigg(\displaystyle\int_{\mathbb{R}^{2d}} (|f|^{q}*|\check{g}|^{q}(x))^{\frac{p}{q}} d\mu_{2d}(x)\bigg)^{\frac{1}{p}}\\
& = & \||f|^{q}*|\check{g}|^{q}\|_{\frac{p}{q},2d}^{\frac{1}{q}}.
\end{eqnarray*}
If $s = \dfrac{2}{q}$ , $t = \dfrac{p}{q}$ and $\dfrac{1}{s} + \dfrac{1}{s'} = 1$ , $\dfrac{1}{t} + \dfrac{1}{t'} = 1$ then $\dfrac{1}{s} + \dfrac{1}{s} = 1 + \dfrac{1}{t}$ and hence as $|f|^{q} , |\check{g}|^{q} \in L^{s}(\mathbb{R}^{2d} , \mathbb{H})$ and with using young inequality we deduce
\begin{equation*}
\| |f|^{q}*|\check{g}|^{q}\|_{t,2d} \leq A_{s}^{4d} A_{t^{'}}^{2d} \| |f|^{q} \|_{s,2d} \||\check{g}|^{q} \|_{s,2d}
\end{equation*}
where $A_{p} = \bigg(\dfrac{p^{\frac{1}{p}}}{q^{\frac{1}{q}}}\bigg)^{\frac{1}{2}}$.\\
However
\begin{equation*}
\| |f|^{q}\|_{s,2d} = \bigg(\displaystyle\int_{\mathbb{R}^{2d}} |f(x)|^{q\frac{2}{q}} d\mu_{2d}(x)\bigg)^{\frac{q}{2}} = \|f\|_{2,2d}^{q}
\end{equation*}
and
\begin{equation*}
\| |\check{g}|^{q} \|_{s,2d} = \bigg( \displaystyle\int_{\mathbb{R}^{2d}} |g(y-x)|^{2} d\mu_{2d}(y)\bigg)^{\frac{q}{2}} = \|g\|_{2,2d}^{q}
\end{equation*}
then
\begin{equation*}
\| |f|^{q}*|\check{g}|^{q}\|_{t,2d}^{\frac{1}{q}} \leq \bigg( A_{s}^{4d} A_{t'}^{2d} \| |f|^{q} \|_{s,2d} \||\check{g}|^{q}\|_{s,2d}\bigg)^{\frac{1}{q}}.
\end{equation*}
We get
\begin{equation*}
\|\mathcal{G}_{g}(f)\|_{p,4d} \leq (A_{s}^{4d} A_{t'}^{2d} \|f\|_{2,2d}^{q} \|g\|_{2,2d}^{q} )^{\frac{1}{q}} = A_{s}^{\frac{4d}{q}} A_{t'}^{\frac{2d}{q}} \|f\|_{2,2d} \|g\|_{2,2d} = C_{p,q} \|f\|_{2,2d} \|g\|_{2,2d}
\end{equation*}
where $C_{p,q} = \bigg(\dfrac{4}{p}\bigg)^{\frac{d}{p}} \bigg(\dfrac{1}{q}\bigg)^{\frac{d}{q}}$ .
\end{proof}

\begin{thm}Let $g_{1}$, $g_{2}$ be two non zero quaternion window functions. Then we have
\begin{enumerate}
\item (Parseval's theorem for $\mathcal{G}_{g_{1}}$ and $\mathcal{G}_{g_{2}}$) For all functions $f_{1}$ and $f_{2}$ in $\in L^{2}(\mathbb{R}^{2d},\mathbb{H})$,
\begin{equation}\label{orthogonalite}
\langle\mathcal{G}_{g_{1}}(f_{1}),\mathcal{G}_{g_{2}}(f_{2})\rangle_{2,4d} =  \langle f_{1} (\overline{g_{1}},\overline{g_{2}})_{2,2d},f_{2}\rangle_{2,2d}.
\end{equation}
\item (Plancherel's theorem for $\mathcal{G}_{g_{1}}$) For every function $f$ in $\in L^{2}(\mathbb{R}^{2d},\mathbb{H})$,
 \begin{equation}\label{orth}
\|\mathcal{G}_{g_{1}}(f)\|_{2,4d} = \|f\|_{2,2d}\|g_{1}\|_{2,2d}.
\end{equation}
\end{enumerate}
In particular, $\mathcal{G}_{g}(f)$ is a linear bounded operator from $L^{2}(\mathbb{R}^{2d},\mathbb{H})$ into $L^{2}(\mathbb{R}^{2d}\times\mathbb{R}^{2d},\mathbb{H})$.\\
 Moreover, if $\|g\|_{2,2d} = 1$, it is isometric.

\end{thm}
\begin{proof}
We assume that $g_{1}$ , $g_{2}$ $\in L^{1}(\mathbb{R}^{2d},\mathbb{H})\cap L^{\infty}(\mathbb{R}^{2d},\mathbb{H})$, according to interpolation, $g_{1}$ , $g_{2}$ $\in L^{2}(\mathbb{R}^{2d},\mathbb{H})$. Then, for every $x \in \mathbb{R}^{2d}$, $f_{1}\overline{T_{x}g_{1}} , f_{2}\overline{T_{x}g_{2}} \in  L^{2}(\mathbb{R}^{2d},\mathbb{H})$ and consequently by using  Fubini and relation (\ref{wq}), we deduce
\begin{eqnarray*}
\displaystyle\iint_{\mathbb{R}^{2d}\times\mathbb{R}^{2d}} |\mathcal{G}_{g_{1}}(f_{1})(x,w)|^{2} d\mu_{2d}(w)  d\mu_{2d}(x)  & = & \displaystyle\int_{\mathbb{R}^{2d}}\bigg(\displaystyle\int_{\mathbb{R}^{2d}} |\mathcal{F}_{Q}(f_{1}\overline{T_{x}g_{1}})(w)|^{2} d\mu_{2d}(w)\bigg)d\mu_{2d}(x)\\
& = & \displaystyle\int_{\mathbb{R}^{2d}}\bigg(\displaystyle\int_{\mathbb{R}^{2d}} |f_{1}(t)\overline{T_{x}g_{1}(t)}|^{2} d\mu_{2d}(t)\bigg)d\mu_{2d}(x)\\
& = & \displaystyle\int_{\mathbb{R}^{2d}}|f_{1}(t)|^{2}\bigg(\displaystyle\int_{\mathbb{R}^{2d}} |\overline{T_{x}g_{1}(t)}|^{2} d\mu_{2d}(x)\bigg)d\mu_{2d}(t)\\
& = & \|f_{1}\|_{2,2d}^{2} \|g_{1}\|_{2,2d}^{2}.
\end{eqnarray*}
Then $\mathcal{G}_{g_{1}}(f_{1}) \in L^{2}(\mathbb{R}^{2d}\times\mathbb{R}^{2d},\mathbb{H})$ and similarly for $\mathcal{G}_{g_{2}}(f_{2})$. By using the Parseval theorem for $\mathcal{F}_{Q}$ we have
\begin{eqnarray}\label{exp}
\langle\mathcal{G}_{g_{1}}(f_{1}),\mathcal{G}_{g_{2}}(f_{2})\rangle_{2,4d} & = & \displaystyle\int_{\mathbb{R}^{2d}} \langle \mathcal{F}_{Q}(f_{1}\overline{T_{x}g_{1}}),\mathcal{F}_{Q}(f_{2}\overline{T_{x}g_{2}})\rangle_{2,2d} d\mu_{2d}(x)\nonumber\\
 & = & \displaystyle\int_{\mathbb{R}^{2d}} \langle f_{1}\overline{T_{x}g_{1}},f_{2}\overline{T_{x}g_{2}}\rangle_{2,2d} d\mu_{2d}(x)\nonumber\\
  & = &  Sc\bigg(\displaystyle\int_{\mathbb{R}^{2d}}\displaystyle\int_{\mathbb{R}^{2d}}
 f_{1}(t) \overline{g_{1}(t-x)} g_{2}(t-x)\overline{f_{2}(t)} d\mu_{2d}(t) d\mu_{2d}(x)\bigg)\nonumber \\
   & = &  Sc\bigg(\displaystyle\int_{\mathbb{R}^{2d}}
 f_{1}(t) \bigg(\displaystyle\int_{\mathbb{R}^{2d}}\overline{g_{1}(\eta)} g_{2}(\eta)d\mu_{2}(\eta)\bigg)\overline{f_{2}(t)} d\mu_{2d}(t)\bigg)\nonumber \\
  & = &  Sc\bigg(\displaystyle\int_{\mathbb{R}^{2d}}  f_{1}(t)(\overline{g_{1}},\overline{g_{2}})_{2,2d} \overline{f_{2}(t)} d\mu_{2d}(t)\bigg) \nonumber\\
  & = & \langle f_{1} (\overline{g_{1}},\overline{g_{2}})_{2,2d},f_{2}\rangle_{2,2d}.
\end{eqnarray}
Let $\varphi$ be defined on  $L^{2}(\mathbb{R}^{2d},\mathbb{H})$ such that  $\varphi(g) = < f_{1} (\overline{g_{1}},\overline{g})_{2,2d},f_{2}>_{2,2d} .$\\
Then, $\varphi$ is linear and by using the Cauchy Schwartz's inequality, we deduce that for every $g \in L^{2}(\mathbb{R}^{2d},\mathbb{H})$
\begin{center}
$|\varphi(g)| \leq \|f_{1}\|_{2,2d}\|f_{2}\|_{2,2d}\|g_{1}\|_{2,2d}\|g\|_{2,2d} = C \|g\|_{2,2d}.$
\end{center}
In particular, $\varphi$ is a linear bounded form on $L^{2}(\mathbb{R}^{2d},\mathbb{H})$.\\
The same, let $\psi$ be the linear form  defined on $L^{1}(\mathbb{R}^{2d},\mathbb{H})\cap L^{\infty}(\mathbb{R}^{2d},\mathbb{H})$ with $\psi(g) = <\mathcal{G}_{g_{1}}(f_{1}),\mathcal{G}_{g}(f_{2})>_{2,4d}.$\\
Then, for every $g \in L^{1}(\mathbb{R}^{2d},\mathbb{H})\cap L^{\infty}(\mathbb{R}^{2d},\mathbb{H}) $ we have
$$\begin{tabular}{lll}
$|\psi(g)|$ & $=$ & $|<\mathcal{G}_{g_{1}}(f_{1}),\mathcal{G}_{g}(f_{2})>_{2,4d}|$\\
& $=$ & $|< f_{1} (\overline{g_{1}},\overline{g})_{2,2d},f_{2}>_{2,2d}|$\\
& $\leq$ & $\|f_{1}\|_{2,2d}\|f_{2}\|_{2,2d}\|g_{1}\|_{2,2d}\|g\|_{2,2d} .$
\end{tabular}$$
Which implies that $\psi$ is a linear form bounded on the dense subset $ L^{1}(\mathbb{R}^{2d},\mathbb{H})\cap L^{\infty}(\mathbb{R}^{2d},\mathbb{H}) $ of $ L^{2}(\mathbb{R}^{2d},\mathbb{H})$.\\
Let $g \in L^{2}(\mathbb{R}^{2d},\mathbb{H})$, there exists a sequence $g_{n}$ $\in  L^{1}(\mathbb{R}^{2d},\mathbb{H})\cap L^{\infty}(\mathbb{R}^{2d},\mathbb{H}) $ such that $\|g - g_{n}\|_{2,2d}\longrightarrow 0$. \\
This shows that $\psi$ is bounded and hence there exists $c\geq 0$ such that for every $n, m \in \mathbb{N}$ we get
\begin{center}
$|\psi(g_{n})-\psi(g_{m})| = |\psi(g_{n} - g_{m})| \leq c \|g_{n} - g_{m}\|_{2,2d}$.
\end{center}
In particular, $(\psi(g_{n}))_{n\in\mathbb{N}}$ is a Cauchy sequence in $\mathbb{H}$, that converges to some number which we call $\psi(g)$. That $\psi(g)$ does not depend on the choice of the sequence $(g_{n})_{n\in\mathbb{N}}$. By using relation (\ref{exp}), $\varphi$ and $\psi$ coincides on $L^{1}(\mathbb{R}^{2d},\mathbb{H})\cap L^{\infty}(\mathbb{R}^{2d},\mathbb{H}) $, and consequently,
\begin{center}
$\psi(g) = \displaystyle\lim_{n\longmapsto +\infty} \psi(g_{n}) = \displaystyle\lim_{n\longmapsto +\infty} \varphi(g_{n}) = \varphi(g)$.
\end{center}

\end{proof}
\begin{cor}(Injection of $\mathcal{G}_{g}$)
\mbox{}\\
Let $g$ be a non zero quaternion windowed function, then for every $f \in L^{2}(\mathbb{R}^{2d},\mathbb{H})$,  if $\mathcal{G}_{g}f = 0$  then  $f = 0$.
\end{cor}
\begin{proof}
Assume that $\mathcal{G}_{g}f = 0$, or we have by Plancherel Formula we have
$$\|\mathcal{G}_{g}(f)\|_{2,4d} = \|f\|_{2,2d}\|g\|_{2,2d}$$
we deduce that $f = 0$  .
\end{proof}
\begin{thm}(Reconstruction formula of $\mathcal{G}_{g}$) \\
Let $g \in L^{2}(\mathbb{R}^{2d},\mathbb{R})$ be a non zero real valued windowed function. Then every quaternion function $f \in L^{2}(\mathbb{R}^{2d},\mathbb{H})$ can be fully reconstructed by
\begin{equation}\label{recon}
f(t) = \dfrac{1}{\|g\|_{2,2d}^{2}}\displaystyle\iint_{\mathbb{R}^{2d}\times\mathbb{R}^{2d}}
 e^{it_{1}w_{1}} \mathcal{G}_{g}f(x,w)g(t-x) e^{jt_{2}w_{2}} d\mu_{2d}(w) d\mu_{2d}(x).
\end{equation}
\end{thm}
\begin{proof}
It follows from (\ref{wq}) that $\mathcal{G}_{g}(f)(x,w) = \mathcal{F}_{Q}(f\overline{T_{x}g})(w)$.\\
Applying now the inverse QFT, we obtain
$$ f(t) \overline{g(t-x)} = \displaystyle\int_{\mathbb{R}^{2d}} e^{it_{1}w_{1}} \mathcal{G}_{g}f(x,w) e^{jt_{2}w_{2}} d\mu_{2d}(w).$$
Multiplying both sides by $g(t-x)$ (which we recall to be real-valued) and integrating with respect to $d\mu_{2d}(x)$ we get
$$f(t) \displaystyle\int_{\mathbb{R}^{2d}}|g(t-x)|^{2}d\mu_{2d}(x) = \displaystyle\iint_{\mathbb{R}^{2d}\times\mathbb{R}^{2d}} e^{it_{1}w_{1}} \mathcal{G}_{g}f(x,w)g(t-x) e^{jt_{2}w_{2}} d\mu_{2d}(w) d\mu_{2d}(x)$$
which gives the desired result.
Then the reconstruction formula can also be rewritten using the kernel of the QWFT as
\begin{center}
$f(t) = \dfrac{1}{\|g\|_{2,2d}^{2}}\displaystyle\iint_{\mathbb{R}^{2d}\times\mathbb{R}^{2d}}
 e^{it_{1}w_{1}} \mathcal{G}_{g}f(x,w)g(t-x) e^{jt_{2}w_{2}} d\mu_{2d}(w) d\mu_{2d}(x)$.
\end{center}
\end{proof}
\ \\ \ \\
\begin{example}
Given two real numbers $a$, $b$ and let $f$ and $g$ be the Gaussian functions defined respectively by\\ $f(x) = (4a)^{\frac{d}{2}} e^{-a|x|^{2}}$ and $g(x) = (4b)^{\frac{d}{2}} e^{-b|x|^{2}}$, then we have\\\\
1. $\|f\|_{2,2d} = \|g\|_{2,2d} = 1$

2. For every $(x,w) \in \mathbb{R}^{2d}\times\mathbb{R}^{2d}$,
 $$|\mathcal{G}_{g}(f)(x,w)|^{2} = \dfrac{(4ab)^{d}}{(a+b)^{2d}}e^{-\frac{2ab}{a+b}|x|^{2}} e^{-\frac{|w|^{2}}{2(a+b)}} $$
\end{example}
\begin{proof}
2.
$$\begin{tabular}{lll}
$\mathcal{G}_{g}(f)(x,w)$ & $=$ & $\displaystyle\int_{\mathbb{R}^{2d}} e^{-it_{1}w_{1}} f(t)\overline{g(t-x)}e^{-jt_{2}w_{2}} d\mu_{2d}(t)$\\
& $=$ & $(16ab)^{\frac{d}{2}} \displaystyle\int_{\mathbb{R}^{2d}} e^{-it_{1}w_{1}}  e^{-a|t|^{2}}e^{-b|t-x|^{2}}e^{-jt_{2}w_{2}} d\mu_{2d}(t)$\\
& $=$ & $(16ab)^{\frac{d}{2}} \displaystyle\int_{\mathbb{R}^{2d}} e^{-it_{1}w_{1}}  e^{-a|t|^{2}} \bigg(e^{-b|t|^{2}} e^{2b<t,x>} e^{-b|x|^{2}} \bigg)e^{-jt_{2}w_{2}} d\mu_{2d}(t)$\\
& $=$ & $(16ab)^{\frac{d}{2}}e^{-b|x|^{2}} e^{\frac{b^{2}}{a+b}|x|^{2}} \displaystyle\int_{\mathbb{R}^{2d}} e^{-it_{1}w_{1}}\bigg(  e^{-(a+b)|t|^{2}}  e^{2b<t,x>} e^{-\frac{b^{2}}{a+b}|x|^{2}}\bigg) e^{-jt_{2}w_{2}} d\mu_{2d}(t)$\\
& $=$ & $(16ab)^{\frac{d}{2}} e^{-\frac{ab}{a+b}|x|^{2}} \displaystyle\int_{\mathbb{R}^{2d}} e^{-it_{1}w_{1}}\bigg(  e^{-(a+b)\big(|t|^{2}-2<t,\frac{b}{a+b}x> + \big|\frac{b}{a+b}x\big|^{2}\big)} \bigg) e^{-jt_{2}w_{2}} d\mu_{2d}(t)$\\
& $=$ & $(16ab)^{\frac{d}{2}} e^{-\frac{ab}{a+b}|x|^{2}} \displaystyle\int_{\mathbb{R}^{2d}} e^{-it_{1}w_{1}} e^{-(a+b)\big|t-\frac{b}{a+b}x\big|^{2}} e^{-jt_{2}w_{2}} d\mu_{2d}(t).$
\end{tabular}$$
By making the change of variable $u = t - \frac{b}{a+b}x$ and applying lemma 3 we get
$$\begin{tabular}{lll}
$\mathcal{G}_{g}(f)(x,w)$ & $=$ & $(16ab)^{\frac{d}{2}} e^{-\frac{ab}{a+b}|x|^{2}} \displaystyle\int_{\mathbb{R}^{2d}} e^{-i\big(u_{1}+\frac{b}{a+b}x_{1}\big)w_{1}} e^{-(a+b)|u|^{2}} e^{-j\big(u_{2}+\frac{b}{a+b}x_{2}\big)w_{1}}d\mu_{2d}(t)$\\
& $=$ & $(16ab)^{\frac{d}{2}} e^{-\frac{ab}{a+b}|x|^{2}}  e^{-i\frac{b}{a+b}x_{1}w_{1}}\bigg( \displaystyle\int_{\mathbb{R}^{2d}} e^{-iu_{1}w_{1}} e^{-(a+b)|u|^{2}} e^{-ju_{2}w_{2}}d\mu_{2d}(t)\bigg)e^{-j\frac{b}{a+b}x_{2}w_{2}}$\\
& $=$ & $(16ab)^{\frac{d}{2}} e^{-\frac{ab}{a+b}|x|^{2}}  e^{-i\frac{b}{a+b}x_{1}w_{1}} \mathcal{F}_{Q}(e^{-(a+b)|u|^{2}})(w) e^{-j\frac{b}{a+b}x_{2}w_{2}}$\\
& $=$ & $\dfrac{(4ab)^{\frac{d}{2}}}{(a+b)^{d}}   e^{-i\frac{b}{a+b}x_{1}w_{1}} e^{-\frac{ab}{a+b}|x|^{2}} e^{-\frac{|w|^{2}}{4(a+b)}} e^{-j\frac{b}{a+b}x_{2}w_{2}} $.\\
\end{tabular}$$
In particular
\begin{center}
$|\mathcal{G}_{g}(f)(x,w)|^{2} = \dfrac{(4ab)^{d}}{(a+b)^{2d}}e^{-\frac{2ab}{a+b}|x|^{2}} e^{-\frac{|w|^{2}}{2(a+b)}} $
\end{center}
\end{proof}
\begin{defn}(The radar quaternion ambiguity function) \cite{1,3}
\mbox{}\\
The radar quaternion ambiguity function (or the cross two sided quaternionic ambiguity function-QAF) of the two-dimensional functions (or signals) $f$ , $g$ $\in L^{2}(\mathbb{R}^{2d},\mathbb{H})$ is denoted by $A(f,g)$ and is defined by
\begin{equation}\label{rq}
A(f,g)(x,w) = \displaystyle\int_{\mathbb{R}^{2d}} e^{-it_{1}w_{1}} f(t+\dfrac{x}{2})\overline{g(t-\dfrac{x}{2})} e^{-jt_{2}w_{2}} d\mu_{2d}(t)
\end{equation}
\end{defn}
The next lemma describes the relationship between the two-sided QWFT and the two-sided QAF mentioned above.
\begin{lem}
\mbox{}\\
Let $f$ , $g$ $\in L^{2}(\mathbb{R}^{2d},\mathbb{H})$ and $g \neq 0$ , for every $(x,w) \in \mathbb{R}^{2d}\times\mathbb{R}^{2d}$, we have
$$A(f,g)(x,w) = e^{iw_{1}\frac{x_{1}}{2}}\hspace{0.1 cm} \mathcal{G}_{g}(f)(x,w)\hspace{0.1 cm} e^{jw_{2}\frac{x_{2}}{2}} .$$
In particular
\begin{equation}\label{p1}
|A(f,g)(x,w)| = |\mathcal{G}_{g}(f)(x,w)| .
\end{equation}
\end{lem}
\begin{defn}(Quaternionic Wigner transformation) \cite{1,3}
\mbox{}\\
The Quaternion Wigner transformation (or the cross two-sided Quaternion Wigner-ville distribution-QWVT) of two-dimensional functions (or signals) $f$ , $g$ $\in L^{2}(\mathbb{R}^{2d},\mathbb{H})$ is given by
\begin{equation}\label{wigner}
W(f,g)(x,w) = \displaystyle\int_{\mathbb{R}^{2d}} e^{-iw_{1}t_{1}} f(x+\dfrac{t}{2}) \overline{g(x-\dfrac{t}{2})} e^{-jw_{2}t_{2}} d\mu_{2d}(t)
\end{equation}
\end{defn}
The following lemma describes the relationship between the two-sided QWFT and the two-sided QWVT mentioned above.

\begin{lem}
\mbox{}\\
Let $f$ , $g$ $\in L^{2}(\mathbb{R}^{2d},\mathbb{H})$  such as $g \neq 0$ and for every $(x,w) \in \mathbb{R}^{2d}\times \mathbb{R}^{2d}$, we have
$$W(f,g)(x,w) =2^{2d} e^{2iw_{1}x_{1}} \mathcal{G}_{\check{g}}(f)(2x,2w) e^{2jw_{2}x_{2}} .$$
In particular
\begin{equation}\label{p2}
|W(f,g)(x,w)| = 2^{2d} |\mathcal{G}_{\check{g}}(f)(2x,2w)|.
\end{equation}
\end{lem}
\section{Uncertainty Principle}
Our purpose in this section is to prove three Lieb uncertainty principle, Logarithmic uncertainty principle, Beckner uncertainty principle in terme of entropy, Heisenberg uncertainty principle for QWFT, QWVT and QAF. Last we study this functions on subset of finite measures.
\subsection{Lieb uncertainty principle for $\mathcal{G}_{\psi}$}
In this subsection we will prove three Lieb uncertainty principle for both of the QWFT, QWVT and QAF.
\begin{defn}
\mbox{}\\
A function $f \in L^{2}(\mathbb{R}^{2d},\mathbb{H})$ is said to be $\varepsilon-$concentrated on a measurable set $U\subseteq \mathbb{R}^{2d}$, where $U^{c} = \mathbb{R}^{2d}\setminus U$, if
\begin{equation}\label{14}
\bigg(\displaystyle\int_{U^{c}} |f(x_{1},x_{2})|^{2} d\mu_{d}(x_{1}) d\mu_{d}(x_{2})\bigg)^{\frac{1}{2}} \leq \varepsilon \|f\|_{2,2d}.
\end{equation}
If $0 \leq \varepsilon \leq \frac{1}{2}$, then the most of energy is concentrated on $U$, and $U$ can be called the essential support of $f$.\\
If $\varepsilon = 0$, then $U$ contains the  support of $f$.
\end{defn}
\begin{thm}(Donoho-Stark)\\
Let $g$ be a (non-zero) quaternion window function and $f \in L^{2}(\mathbb{R}^{2d},\mathbb{H})$ such that $f \neq 0$. Let $U$ a measurable set of $\mathbb{R}^{2d}\times\mathbb{R}^{2d}$ and $\varepsilon\geq 0$.\\
 If $\mathcal{G}_{g}(f)$ is $\varepsilon-$concentrated on $U$, hence we have
 \begin{equation}\label{donohostarck}
 (1-\varepsilon^{2}) \leq \mu_{4d}(U).
 \end{equation}
\end{thm}
\begin{proof}
By relation (\ref{orth}), we have
\begin{eqnarray*}
\|f\|_{2,2d}^{2}\|g\|_{2,2d}^{2} & = & \|\mathcal{G}_{g}(f)\|_{2,4d}^{2}\\
& = & \|\chi_{U^{c}}\mathcal{G}_{g}(f)\|_{2,4d}^{2} + \|\chi_{U}\mathcal{G}_{g}(f)\|_{2,4d}^{2}\\
& \leq & \varepsilon^{2}\|\mathcal{G}_{g}(f)\|_{2,4d}^{2} +  \|\chi_{U}\mathcal{G}_{g}(f)\|_{2,4d}^{2}\\
& = & \varepsilon^{2} \|f\|_{2,2d}^{2}\|g\|_{2,2d}^{2} +  \|\chi_{U}\mathcal{G}_{g}(f)\|_{2,4d}^{2}.
\end{eqnarray*}
Consequently , by using the relation (\ref{1}), we deduce
\begin{eqnarray*}
(1-\varepsilon^{2}) \|f\|_{2,2d}^{2}\|g\|_{2,2d}^{2}  \leq  \|\chi_{U}\mathcal{G}_{g}(f)\|_{2,4d}^{2}
 \leq  \mu_{4d}(U) \|\mathcal{G}_{g}(f)\|_{\infty,4d}^{2} \leq \mu_{4d}(U) \|f\|_{2,2d}^{2}\|g\|_{2,2d}^{2}.
\end{eqnarray*}
we may simplify by $\|f\|_{2,2d}^{2}\|g\|_{2,2d}^{2}$ to obtain the desired result.
\end{proof}
\\
From the above definition and Lieb inequality, the following theorem follows
\begin{thm} (Lieb uncertainty principle)\\
Let $g$ be (non-zero) quaternion window function and $f \in L^{2}(\mathbb{R}^{2d},\mathbb{H})$ such that $f \neq 0$. Let $U$ a measurable set of $\mathbb{R}^{2d}\times\mathbb{R}^{2d}$ and $\varepsilon \geq 0$. If  $\mathcal{G}_{g}(f)$ is $\varepsilon$-concentrated on $U$, hence for every $p > 2$ we have
\begin{equation}\label{16}
C_{p,q}^{\frac{2p}{2-p}} \big(1 - \varepsilon^{2}\big)^{\frac{p}{p-2}} \leq  \mu_{4d}(U).
\end{equation}
when $C_{p,q} = \bigg(\dfrac{4}{p}\bigg)^{\frac{d}{p}}  \bigg(\dfrac{1}{q}\bigg)^{\frac{d}{q}}$; $\dfrac{1}{p}+\dfrac{1}{q} = 1$.
\end{thm}
\begin{proof}
 $\mathcal{G}_{g}(f)$ is $\varepsilon$-concentrated on $U$, that is to say
$$\|\chi_{U^{c}}\mathcal{G}_{g}(f)\|_{2,4d} \leq \varepsilon \|\mathcal{G}_{g}(f)\|_{2,4d} = \varepsilon \|f\|_{2,2d} \|g\|_{2,2d} .$$
So
\begin{equation}\label{17}
\begin{tabular}{lllll}
$\|\chi_{U} \mathcal{G}_{g}(f)\|_{2,4d}^{2}$ & $=$ & $\| \mathcal{G}_{g}(f)\|_{2,4d}^{2} - \|\chi_{U^{c}} \mathcal{G}_{g}(f)\|_{2,4d}^{2}$ & $\geq$ & $\|f\|_{2,2d}^{2} \|g\|_{2,2d}^{2} - \varepsilon^{2} \|f\|_{2,2d}^{2} \|g\|_{2,2d}^{2}$\\
& $=$ & $\|f\|_{2,2d}^{2} \|g\|_{2,2d}^{2} (1 - \varepsilon^{2})$
\end{tabular}
\end{equation}
and
$$\begin{tabular}{lllll}
$\|\chi_{U} \mathcal{G}_{g}(f)\|_{2,4d}^{2}$ & $=$ & $\displaystyle\iint_{\mathbb{R}^{2d}\times\mathbb{R}^{2d}} \chi_{U}(x,w) |\mathcal{G}_{g}(f)(x,w)|^{2} d\mu_{4d}(x,w)$\\
 & $\leq$ & $\bigg(\displaystyle\iint_{\mathbb{R}^{2d}\times\mathbb{R}^{2d}} \big(\chi_{U}(x,w)\big)^{\frac{p}{p-2}} d\mu_{4d}(x,w)\bigg)^{\frac{p-2}{p}} $\\
& $\times$ & $ \bigg(\displaystyle\iint_{\mathbb{R}^{2d}\times\mathbb{R}^{2d}} |\mathcal{G}_{g}(f)(x,w)|^{2\frac{p}{2}} d\mu_{4d}(x,w)\bigg)^{\frac{2}{p}}$\\
 & $=$ & $\big( \mu_{4d}(U)\big)^{\frac{p-2}{p}} \times \|\mathcal{G}_{g}(f)\|_{p,4d}^{2}.$
\end{tabular}$$
On the other hand, and again by Lieb inequality and the relation (\ref{17}), we deduce that
$$\|\chi_{U} \mathcal{G}_{g}(f)\|_{2,4d}^{2} \leq \big( \mu_{4d}(U)\big)^{\frac{p-2}{p}} \times \|\mathcal{G}_{g}(f)\|_{p,4d}^{2} \leq C_{p,q}^{2}\big(\mu_{4d}(U)\big)^{\frac{p-2}{p}} \|f\|_{2,2d}^{2} \|g\|_{2,2d}^{2} $$
then
$$\|f\|_{2,2d}^{2} \|g\|_{2,2d}^{2} (1 - \varepsilon^{2}) \leq C_{p,q}^{2} \big(\mu_{4d}(U)\big)^{\frac{p-2}{p}} \|f\|_{2,2d}^{2} \|g\|_{2,2d}^{2}$$
and consequently
$$(1 - \varepsilon^{2}) \leq C_{p,q}^{2} \big(\mu_{4d}(U)\big)^{\frac{p-2}{p}}$$
hence
$$C_{p,q}^{\frac{2p}{2-p}} \big(1 - \varepsilon^{2}\big)^{\frac{p}{p-2}} \leq  \mu_{4d}(U).$$
\end{proof}
\begin{cor}
\mbox{}\\
Let $f$ , $g$ $\in L^{2}(\mathbb{R}^{2d},\mathbb{H})$ such that $f \neq 0$ and $g \neq 0$. $U \subseteq \mathbb{R}^{2d}\times\mathbb{R}^{2d}$ a measurable set and $\varepsilon \geq 0$. If $A(f,g)$ is $\varepsilon$-concentrated on $U$ we get
\begin{center}
$\forall p > 2$ \hspace{0.1 cm} $\mu_{4d}(U) \geq (1 - \varepsilon^{2})^{\frac{p}{p-2}} \bigg(C_{p,q}\bigg)^{\frac{2p}{2-p}}$.
\end{center}
\end{cor}
\begin{proof}
By relation (\ref{p1}), we have
$$\|\chi_{U} A(f,g)\|_{2,4d}^{2} = \displaystyle\iint_{U} |A(f,g)(x,w)|^{2} d\mu_{4d}(x,w)
= \displaystyle\iint_{U} |\mathcal{G}_{g}(f)(x,w)|^{2} d\mu_{4d}(x,w) = \|\chi_{U} \mathcal{G}_{g}(f)\|_{2,4d}^{2}$$
and
$$\|\chi_{U} \mathcal{G}_{g}(f)\|_{2,4d} \geq \|f\|_{2,2d}^{2}\|g\|_{2,2d}^{2} (1 - \varepsilon^{2})$$
then by relation (\ref{16}), we deduce that
$$\mu_{4d}(U) \geq (1 - \varepsilon^{2})^{\frac{p}{p-2}}\bigg(C_{p,q}\bigg)^{\frac{2p}{2-p}}.$$
\end{proof}
\begin{lem}
Let $U \subseteq \mathbb{R}^{d}$ , $c > 0$ and $V = \{cx| x \in U , c>0\}$ then
\begin{equation}\label{18}
\mu_{d}(V) = c^{d} \mu_{d}(U).
\end{equation}
\end{lem}
\begin{proof}
$$\mu_{d}(V) = \displaystyle\int_{\mathbb{R}^{d}} \chi_{V}(x) d\mu_{d}(x) = \displaystyle\int_{\mathbb{R}^{d}} \chi_{U}\big(\frac{x}{c}\big) d\mu_{d}(x) = c^{d} \displaystyle\int_{\mathbb{R}^{d}} \chi_{U}(y) d\mu_{d}(y) = c^{d} \mu_{d}(U).$$
\end{proof}
\begin{cor}
Let $f$,$g$ $\in L^{2}(\mathbb{R}^{2d},\mathbb{H})$ such that $f\neq0$ and $g\neq0$. $U\subseteq \mathbb{R}^{2d}\times\mathbb{R}^{2d}$ a measurable set and $\varepsilon \geq 0$. If $W(f,g)$ is $\varepsilon$-concentrated on $U$ we get
\begin{center}
$\forall p > 2$ \hspace{0.5 cm} $\mu_{4d}(U) \geq \dfrac{1}{2^{4d}} (1-\varepsilon^{2})^{\frac{p}{p-2}} \bigg(C_{p,q}\bigg)^{\frac{2p}{2-p}}$.
\end{center}
\end{cor}
\begin{proof}
We have
$$\begin{tabular}{lllll}
$\|\chi_{U} W(f,g)\|_{2,4d}^{2}$ & $=$ & $\displaystyle\iint_{U} |W(f,g)(x,w)|^{2} d\mu_{4d}(x,w)$ & $=$ & $2^{4d} \displaystyle\iint_{U} |\mathcal{G}_{\check{g}}(f)(2x,2w)|^{2} d\mu_{4d}(x,w)$\\
& $=$ & $\displaystyle\iint_{V} |\mathcal{G}_{\check{g}}(f)(y,\lambda)|^{2} d\mu_{4d}(y,\lambda)$ & $=$ & $\|\chi_{V} \mathcal{G}_{\check{g}}(f)\|_{2,4d}^{2}$
\end{tabular}$$
where
$$V = \{(y,\lambda) \in \mathbb{R}^{2d}\times\mathbb{R}^{2d}|\big(\frac{y}{2},\frac{\lambda}{2}\big)\in U\} = \{(2x,2w)\in \mathbb{R}^{2d}\times\mathbb{R}^{2d}|(x,w)\in U\}$$
$W(f,g)$ is $\varepsilon$-concentrated on $U$ then $\mathcal{G}_{g}(f)$ is $\varepsilon$-concentrated on $V$ and
$$\mu_{4d}(V) \geq  (1-\varepsilon^{2})^{\frac{p}{p-2}} \bigg(C_{p,q}\bigg)^{\frac{2p}{2-p}}$$
using (\ref{18}) we get $\mu_{4d}(V) = 2^{4d} \mu_{4d}(U)$ hence
\begin{center}
$\mu_{4d}(U) \geq \dfrac{1}{2^{4d}} (1-\varepsilon^{2})^{\frac{p}{p-2}} \bigg(C_{p,q}\bigg)^{\frac{2p}{2-p}}$
\end{center}
\end{proof}
\begin{thm}
Let $f$ , $g \in L^{2}(\mathbb{R}^{2d},\mathbb{H})$ , then
\begin{equation}\label{200}
\mu_{4d}(supp(\mathcal{G}_{g}(f))) \geq C_{p,q}^{\frac{2p}{2-p}}.
\end{equation}
\end{thm}
\begin{proof}
Because
\begin{center}
$\|\chi_{supp(G_{g}(f))^{c}}G_{g}f\|_{2,4d} = 0 \leq 0 \times \|\chi_{supp(G_{g}(f))}G_{g}f\|_{2,4d}$
\end{center}
then $G_{g}f$ is a $0 - concentrated$ on $supp(G_{g}f)$ then  an $\varepsilon = 0$ we get
\begin{center}
$\mu_{4d}(supp(\mathcal{G}_{g}(f))) \geq C_{p,q}^{\frac{2p}{2-p}}$
\end{center}
\end{proof}
\begin{cor}
Let $f$ , $g \in L^{2}(\mathbb{R}^{2d},\mathbb{H})$ , then
\begin{center}
$\mu_{4d}(supp(A(f,g))) \geq C_{p,q}^{\frac{2p}{2-p}}$.
\end{center}
\end{cor}
\begin{proof}
By relation (\ref{p1}) we get $supp(G_{g}f)) = supp(A(f,g))$ .
\end{proof}
\begin{cor}
Let $f$ , $g \in L^{2}(\mathbb{R}^{2d},\mathbb{H})$ , then
\begin{center}
$\mu_{4d}(supp(W(f,g))) \geq \dfrac{1}{2^{4d}} \bigg(C_{p,q}\bigg)^{\frac{2p}{2-p}} $.
\end{center}
\end{cor}
\begin{proof}
Let $(x,w) \in supp(W(f,g))$ then $G_{g}f(2x,2w) \neq 0$, and consequently by (\ref{18}) , we get
\begin{center}
$\mu_{4d}(supp(W(f,g))) = \dfrac{1}{2^{4d}} \mu_{4d}(supp(G_{g}f))$
\end{center}
then by (\ref{200}) ,
\begin{center}
$\mu_{4d}(supp(G_{g}f))) \geq \dfrac{1}{2^{4d}} \bigg(C_{p,q}\bigg)^{\frac{2p}{2-p}} $.
\end{center}
\end{proof}

\subsection{Logarithmic uncertainty principle for $\mathcal{G}_{\psi}$}
In this subsection we used  the theorem \ref{logarithm}  to obtain the Logarithmic uncertainty principle for the multivariate two sided QWFT.
\begin{thm} (Logarithmic uncertainty principle for the WQFT)\\
Let $f$, $g$ $\in \mathcal{S}(\mathbb{R}^{2d},\mathbb{H})$ then\\
$ D_{2d} \|f\|_{2,2d}^{2} \|g\|_{2,2d}^{2} \leq$
\begin{equation}\label{logarithmic}
\displaystyle\iint_{\mathbb{R}^{2d}\times\mathbb{R}^{2d}} ln|w| |\mathcal{G}_{g}(f)(x,w)|^{2} d\mu_{2d}(x)\hspace{0.1 cm} d\mu_{2d}(w) + \|g\|_{2,2d}^{2} \displaystyle\int_{\mathbb{R}^{2d}} ln|t| |f(t)|^{2} d\mu_{2d}(t),
\end{equation}
where $D_{2d}$ is given by (\ref{Dd}).
\end{thm}
\begin{proof}
Notice that $f$, $g \in \mathcal{S}(\mathbb{R}^{2d},\mathbb{H})$, this implies that $f\overline{T_{x}g} \in \mathcal{S}(\mathbb{R}^{2d},\mathbb{H})$ . Therefore, we may replace $f$ by $f\overline{T_{x}g}$ on both side of (\ref{loga}) and get
$$\displaystyle\int_{\mathbb{R}^{2d}} ln|w| |\mathcal{F}_{Q}(f\overline{T_{x}g})(w)|^{2} d\mu_{2d}(w) + \displaystyle\int_{\mathbb{R}^{2d}} ln|t| |f\overline{T_{x}g(t)}|^{2} d\mu_{2d}(t) \geq D \displaystyle\int_{\mathbb{R}^{2d}} |f\overline{T_{x}g}(t)|^{2}d\mu_{2d}(t)$$
integration both sides of this equation with respect to $d\mu_{2d}(x)$ yields
\begin{center}
$\displaystyle\iint_{\mathbb{R}^{2d}\times\mathbb{R}^{2d}} ln|w| |\mathcal{F}_{Q}(f\overline{T_{x}g})(w)|^{2} d\mu_{2d}(w)\hspace{0.1 cm} d\mu_{2d}(x) + \displaystyle\iint_{\mathbb{R}^{2d}\times\mathbb{R}^{2d}} ln|t| |f\overline{T_{x}g}(t)|^{2} d\mu_{2d}(t)\hspace{0.1 cm} d\mu_{2d}(x)$\\
$ \geq D \displaystyle\iint_{\mathbb{R}^{2d}\times\mathbb{R}^{2d}} |f\overline{T_{x}g}(t)|^{2}d\mu_{2d}(t)\hspace{0.1 cm} d\mu_{2d}(x)$
\end{center}
we obtain
\begin{center}
$\displaystyle\iint_{\mathbb{R}^{2d}\times\mathbb{R}^{2d}} ln|w| |\mathcal{G}_{g}(f)(x,w)|^{2} d\mu_{2d}(w)\hspace{0.1 cm} d\mu_{2d}(x) + \displaystyle\iint_{\mathbb{R}^{2d}\times\mathbb{R}^{2d}} ln|t| |f(t)|^{2} |g(t-x)|^{2} d\mu_{2d}(t)\hspace{0.1 cm} d\mu_{2d}(x)$\\ $\geq D \displaystyle\int_{\mathbb{R}^{2d}} |f(t)|^{2} d\mu_{2d}(t) \displaystyle\int_{\mathbb{R}^{2d}} |g(x)|^{2} d\mu_{2d}(x)$
\end{center}
finally, we have
$$\displaystyle\iint_{\mathbb{R}^{2d}\times\mathbb{R}^{2d}} ln|w| |\mathcal{G}_{g}(f)(x,w)|^{2} d\mu_{2d}(x)\hspace{0.1 cm} d\mu_{2d}(w) + \|g\|_{2,2d}^{2} \displaystyle\int_{\mathbb{R}^{2d}} ln|t| |f(t)|^{2} d\mu_{2d}(t) \geq D\hspace{0.1 cm} \|f\|_{2,2d}^{2} \|g\|_{2,2d}^{2}.$$
\end{proof}

\begin{cor}
Let $f$, $g$ $\in \mathcal{S}(\mathbb{R}^{2d},\mathbb{H})$ then
$$\displaystyle\iint_{\mathbb{R}^{2d}\times\mathbb{R}^{2d}} ln|w| |A(f,g)(x,w)|^{2} d\mu_{2d}(x)\hspace{0.1 cm} d\mu_{2d}(w) + \|g\|_{2,2d}^{2} \displaystyle\int_{\mathbb{R}^{2d}} ln|t| |f(t)|^{2} d\mu_{2d}(t) \geq D\hspace{0.1 cm} \|f\|_{2,2d}^{2} \|g\|_{2,2d}^{2}.$$
\end{cor}
\begin{cor}
Let $f$, $g$ $\in \mathcal{S}(\mathbb{R}^{2d},\mathbb{H})$ then
$$\displaystyle\iint_{\mathbb{R}^{2d}\times\mathbb{R}^{2d}} ln|w| |W(f,g)(x,w)|^{2} d\mu_{2d}(x)\hspace{0.1 cm} d\mu_{2d}(w) + \|g\|_{2,2d}^{2} \displaystyle\int_{\mathbb{R}^{2d}} ln|t| |f(t)|^{2} d\mu_{2d}(t) \geq  \big(D - ln(2)\big) \hspace{0.1 cm} \|f\|_{2,2d}^{2} \|g\|_{2,2d}^{2}.$$
\end{cor}
\begin{proof}
We have
$$\begin{tabular}{lll}
$\displaystyle\int_{\mathbb{R}^{2d}\times\mathbb{R}^{2d}} ln|w| |W(f,g)(x,w)|^{2} d\mu_{2d}(x) d\mu_{2d}(w)$ & $=$ & $\displaystyle\int_{\mathbb{R}^{2d}\times\mathbb{R}^{2d}} 2^{4d} ln|w| |\mathcal{G}_{\check{g}}(2x,2w)|^{2} d\mu_{2d}(x) d\mu_{2d}(w)$\\
& $=$ & $\displaystyle\int_{\mathbb{R}^{2d}\times\mathbb{R}^{2}} ln\bigg|\frac{\sigma}{2}\bigg| |\mathcal{G}_{\check{g}}(f)(y,\sigma)|^{2} d\mu_{2d}(y) d\mu_{2d}(\sigma)$\\
& $=$ & $\displaystyle\int_{\mathbb{R}^{2d}\times\mathbb{R}^{2d}} ln|\sigma| |\mathcal{G}_{\check{g}}(f)(y,\sigma)|^{2} d\mu_{2d}(y) d\mu_{2d}(\sigma) $\\
& $-$ & $ \displaystyle\int_{\mathbb{R}^{2d}\times\mathbb{R}^{2d}} ln(2) |\mathcal{G}_{\check{g}}(f)(y,\sigma)|^{2} d\mu_{2d}(y) d\mu_{2d}(\sigma)$\\
& $=$ & $\displaystyle\int_{\mathbb{R}^{2d}\times\mathbb{R}^{2d}} ln|\sigma| |\mathcal{G}_{\check{g}}(f)(y,\sigma)|^{2} d\mu_{2d}(y) d\mu_{2d}(\sigma)$\\
& $ - $ & $ ln(2) \|f\|_{2,2d}^{2}  \|g\|_{2,2d}^{2}.$
\end{tabular}$$
Using the inequality
$$\displaystyle\iint_{\mathbb{R}^{2d}\times\mathbb{R}^{2d}} ln|w| |\mathcal{G}_{g}(f)(x,w)|^{2} d\mu_{2d}(x)\hspace{0.1 cm} d\mu_{2d}(w) + \|g\|_{2,2d}^{2} \displaystyle\int_{\mathbb{R}^{2d}} ln|t| |f(t)|^{2} d\mu_{2d}(t) \geq D\hspace{0.1 cm} \|f\|_{2,2d}^{2} \|g\|_{2,2d}^{2}.$$
we get
$$\displaystyle\iint_{\mathbb{R}^{2d}\times\mathbb{R}^{2d}} ln|w| |W(f,g)(x,w)|^{2} d\mu_{2d}(x)\hspace{0.1 cm} d\mu_{2d}(w) + \|g\|_{2,2d}^{2} \displaystyle\int_{\mathbb{R}^{2d}} ln|t| |f(t)|^{2} d\mu_{2d}(t) \geq (D - ln(2))\hspace{0.1 cm} \|f\|_{2,2d}^{2} \|g\|_{2,2d}^{2}.$$
\end{proof}

\subsection{The Beckner's uncertainty principle in terms of entropy for $\mathcal{G}_{\psi}$}
Clearly the entropy represents an advantageous way to measure the decay of a function $f$, so that it was very interesting to localize the entropy of a probability measure and its quaternion Fourier transform.
\begin{defn} (Entropy)\\
The entropy of a probability density function $P$ on $\mathbb{R}^{2d}\times\mathbb{R}^{2d}$ is defined by
$$E(P) = - \displaystyle\iint_{\mathbb{R}^{2d}\times\mathbb{R}^{2d}} ln(P(x,w)) P(x,w) d\mu_{4d}(x,w).$$
\end{defn}
The aim of the following is to generalize the localization  of the entropy to the QWFT over the quaternion windowed plane, the quaternion radar ambiguity function and the quaternion Wigner-Vile transform.
\begin{thm}
Let $g$ be a quaternionic windowed function and $f \in L^{2}(\mathbb{R}^{2d},\mathbb{H})$ with $f, g \neq 0,$  then
\begin{equation}\label{10}
E(|\mathcal{G}_{g}(f)|^{2}) \geq \|f\|_{2,2d}^{2} \|g\|_{2,2d}^{2}(2\hspace{0.1 cm}d\hspace{0.1 cm}ln(2) - ln(\|f\|_{2,2d}^{2} \|g\|_{2,2d}^{2})) .
\end{equation}
\end{thm}
\begin{proof}
Assume that $\|f\|_{2,2d} = \|g\|_{2,2d} = 1$ , then by relation (\ref{1}) we deduce that
$$\forall (x,w) \in \mathbb{R}^{2d}\times\mathbb{R}^{2d}  \hspace{0.1 cm} , \hspace{0.1 cm} |\mathcal{G}_{g}(f)(x,w)| \leq \|\mathcal{G}_{g}(f)\|_{\infty,4d} \leq \|f\|_{2,2d} \|g\|_{2,2d} = 1$$
then $ln (|\mathcal{G}_{g}(f)|) \leq 0$ in particular $E(|\mathcal{G}_{g}(f)|)\geq 0$.\\
$\bullet$ Therefore if the entropy $E(|\mathcal{G}_{g}(f)|) = +\infty $ then the inequality (\ref{10}) hold trivially .\\
$\bullet$ Suppose now that the entropy $E(|\mathcal{G}_{g}(f)|) < +\infty $  and let $0 < x < 1$ and $H_{x}$ be the function defined on $]2,3]$ by
$$H_{x}(p) = \dfrac{x^{p} - x^{2}}{p - 2}$$
then
\begin{center}
$\forall p \in ]2,3]$, $\dfrac{dH_{x}}{dp}(p) = \dfrac{(p-2) x^{p} ln(x) - x^{p} + x^{2}}{(p-2)^{2}}$.
\end{center}
The sign of $\dfrac{dH_{x}}{dp}$ is the same as that of the function $K_{x}(p) = (p-2) x^{p} ln(x) - x^{p} + x^{2}$.\\
For every $0 < x < 1$, the function $K_{x}$ is differentiable on $\mathbb{R}$, especially on $]2,3]$,\\
 and its derivative is
$$\dfrac{dK_{x}}{dp}(p) = (p-2) (ln(x))^{2} x^{p} .$$
We have that, for all $0< x < 1$, $\dfrac{dK_{x}}{dp}(p)$ is positive on $]2,3]$, then $K_{x}$ is increasing on $]2,3]$.\\
For all $0<x<1$, $\displaystyle\lim_{p \longmapsto 2^{+}} K_{x}(p) = K_{x}(2) = 0$ then $K_{x}$ is positive which implies that $\dfrac{dH_{x}}{dp}$ is positive also on $]2,3]$ and consequently $p \mapsto H_{x}(p)$  is increasing on $]2,3]$.
In particular,
\begin{center}
$\forall p \in ]2,3]$, \hspace{0.5 cm} $x^{2} ln(x) = \displaystyle\lim_{p\longrightarrow 2^{+}} H_{x}(p) \leq \dfrac{x^{p} - x^{2}}{p - 2}$
\end{center}
hence
\begin{equation}\label{11}
\forall p \in ]2,3], \hspace{0.5 cm} 0 \leq \dfrac{ x^{2} - x^{p}}{p - 2} \leq -x^{2} ln(x).
\end{equation}
Inequality (\ref{11}) holds true for $x = 0$ and $x = 1$. Hence for every $0 \leq x \leq 1$ we have
$$\forall p \in ]2,3], \hspace{0.5 cm} 0 \leq \dfrac{x^{2} - x^{p}}{p - 2} \leq -x^{2} ln(x).$$
We have already observed that all every $(x,w) \in \mathbb{R}^{2d}\times\mathbb{R}^{2d}$, $0\leq |\mathcal{G}_{g}(f)(x,w)|^{2} \leq 1$; then we get for every $p \in ]2,3]$
\begin{equation}\label{12}
0 \leq \dfrac{|\mathcal{G}_{g}(f)(x,w)|^{2} - |\mathcal{G}_{g}(f)(x,w)|^{p}}{p - 2} \leq - |\mathcal{G}_{g}(f)(x,w)|^{2} ln(|\mathcal{G}_{g}(f)(x,w)|).
\end{equation}
Let $\varphi$ be the function defined on $[2, +\infty [$ by
\begin{center}
$\varphi(p) = \bigg(\displaystyle\iint_{\mathbb{R}^{2d}\times\mathbb{R}^{2d}} |\mathcal{G}_{g}(f)(x,w)|^{p} d\mu_{4d}(x,w)\bigg) - \big(C_{p,q}\big)^{p} .$
\end{center}
According to Lieb inequality, we know that for every $2 \leq p < +\infty$ the QWFT $\mathcal{G}_{g}(f)$ belongs to $L^{p}(\mathbb{R}^{2d}\times\mathbb{R}^{2d},\mathbb{H})$ and we have
\begin{equation}\label{13}
\|\mathcal{G}_{g}(f)\|_{p,4d} \leq C_{p,q} \|f\|_{2,2d} \|g\|_{2,2d}\hspace{0.2 cm} ,\hspace{0.2 cm} C_{p,q} = \bigg(\dfrac{4}{q}\bigg)^{\frac{d}{q}} \bigg(\dfrac{1}{p}\bigg)^{\frac{d}{p}}.
\end{equation}
Then, relation (\ref{13}) implies that $\varphi(p)\leq 0$ for every $p \in [2, +\infty[$ and by Plancherel's theorem we have
$$\varphi(2) = \displaystyle\iint_{\mathbb{R}^{2d}\times\mathbb{R}^{2d}} |\mathcal{G}_{g}(f)(x,w)|^{2} d\mu_{4d}(x,w) - 1 = \|\mathcal{G}_{g}(f)\|_{2,4d}^{2} - 1 = \|f\|_{2,2d}^{2}\|g\|_{2,2d}^{2} - 1 = 0$$.
Therefore $\bigg(\dfrac{d\varphi}{dp}\bigg)_{p = 2^{+}} \leq 0$ whenever this derivative is well defined. On the other hand, we have for every $p \in ]2,3]$ and for $(x,w) \in \mathbb{R}^{2d}\times\mathbb{R}^{2d}$
\begin{center}
$\bigg|\dfrac{|\mathcal{G}_{g}(f)(x,w)|^{p} - |\mathcal{G}_{g}(f)(x,w)|^{2} }{p-2}\bigg| \leq - |\mathcal{G}_{g}(f)(x,w)|^{2} ln(|\mathcal{G}_{g}(f)(x,w)|)$.
\end{center}
Then\\
$\displaystyle\iint_{\mathbb{R}^{2d}\times\mathbb{R}^{2d}}\bigg|\dfrac{|\mathcal{G}_{g}(f)(x,w)|^{p} - |\mathcal{G}_{g}(f)(x,w)|^{2} }{p-2}\bigg| d\mu_{4d}(x,w)$
\begin{eqnarray*}
 & \leq & -\displaystyle\iint_{\mathbb{R}^{2d}\times\mathbb{R}^{2d}} |\mathcal{G}_{g}(f)(x,w)|^{2} ln(|\mathcal{G}_{g}(f)(x,w)|)d\mu_{4d}(x,w)\\
& = & - \dfrac{1}{2} \displaystyle\iint_{\mathbb{R}^{2d}\times\mathbb{R}^{2d}} |\mathcal{G}_{g}(f)(x,w)|^{2} ln(|\mathcal{G}_{g}(f)(x,w)|^{2})d\mu_{4d}(x,w)\\
& = & \dfrac{1}{2} E(|\mathcal{G}_{g}(f)(x,w)|^{2}) < +\infty.
\end{eqnarray*}
Moreover, for every $p \in ]3, +\infty[$ and for every $(x,w) \in \mathbb{R}^{2d}\times\mathbb{R}^{2d}$, then
$$\bigg|\dfrac{|\mathcal{G}_{g}(f)(x,w)|^{p} - |\mathcal{G}_{g}(f)(x,w)|^{2} }{p-2}\bigg| \leq 2|\mathcal{G}_{g}(f)(x,w)|^{2}$$
and consequently
$$\begin{tabular}{lll}
$\displaystyle\iint_{\mathbb{R}^{2d}\times\mathbb{R}^{2d}}\bigg|\dfrac{|\mathcal{G}_{g}(f)(x,w)|^{p} - |\mathcal{G}_{g}(f)(x,w)|^{2} }{p-2}\bigg| d\mu_{4d}(x,w)$ & $\leq$ & $2\|\mathcal{G}_{g}(f)\|^{2}_{2,4d} = 2 < +\infty .$\\
\end{tabular}$$
Using relation (\ref{12}) and Lebesgue's dominated convergence theorem we have\\
$\bigg(\dfrac{d}{dp}\displaystyle\iint_{\mathbb{R}^{2d}\times\mathbb{R}^{2d}} |\mathcal{G}_{g}(f)(x,w)|^{p} d\mu_{4d}(x,w)\bigg)_{p=2^{+}} $
\begin{eqnarray*}
& = &  \displaystyle\lim_{p\longrightarrow2^{+}} \displaystyle\iint_{\mathbb{R}^{2d}\times\mathbb{R}^{2d}} \dfrac{|\mathcal{G}_{g}(f)(x,w)|^{p}-|\mathcal{G}_{g}(f)(x,w)|^{2}}{p-2}d\mu_{4d}(x,w) \\
& = &  \displaystyle\iint_{\mathbb{R}^{2d}\times\mathbb{R}^{2d}} \displaystyle\lim_{p\longrightarrow2^{+}}\dfrac{|\mathcal{G}_{g}(f)(x,w)|^{p}-|\mathcal{G}_{g}(f)(x,w)|^{2}}{p-2}d\mu_{4d}(x,w) \\
& = &  \displaystyle\iint_{\mathbb{R}^{2d}\times\mathbb{R}^{2d}} ln(|\mathcal{G}_{g}(f)(x,w)|) |\mathcal{G}_{g}(f)(x,w)|^{2} d\mu_{4d}(x,w)\\
& = &  \dfrac{1}{2} \displaystyle\iint_{\mathbb{R}^{2d}\times\mathbb{R}^{2d}} ln(|\mathcal{G}_{g}(f)(x,w)|^{2}) |\mathcal{G}_{g}(f)(x,w)|^{2} d\mu_{4d}(x,w)\\
& =  &  - \dfrac{1}{2} E(|\mathcal{G}_{g}(f)(x,w)|^{2}),
\end{eqnarray*}
and consequently
\begin{center}
$\bigg(\dfrac{d\varphi}{dp}\bigg)_{p=2^{+}} = - \dfrac{1}{2} E(|\mathcal{G}_{g}(f)|^{2}) - \bigg(\dfrac{d\big(\big(C_{p,q}\big)^{p}\big)}{dp}\bigg)_{p=2^{+}}$ .
\end{center}
On the other hand,
$$\hspace{-1.5 cm}\begin{tabular}{lll}
$\bigg(\dfrac{d(C_{p,q})^{p}}{dp}\bigg)_{p=2^{+}}$ & $ = $ & $ \bigg(d\bigg[-\dfrac{4}{p^{2}}\bigg(\dfrac{p-1}{p}\bigg)^{p-1} + \dfrac{4}{p}\bigg(\dfrac{p-1}{p}\bigg)^{p-1}\bigg(ln\big(\dfrac{p-1}{p}\big)+\dfrac{1}{p}\bigg)\bigg]\bigg[\bigg(\dfrac{4}{p}\bigg)^{d-1}\bigg(\dfrac{p-1}{p}\bigg)^{(p-1)(d-1)}\bigg]\bigg)_{p=2^{+}}$\\
& $ = $ & $-d$ $ln(2)$
\end{tabular}$$
so
\begin{center}
$\bigg(\dfrac{d\varphi}{dp}\bigg)_{p=2^{+}} = - \dfrac{1}{2} E(|\mathcal{G}_{g}(f)|^{2}) + $ $d$ $ln(2) \leq 0$
\end{center}
which gives
\begin{center}
$E(|\mathcal{G}_{g}(f)|^{2}|) \geq $ $2$ $d$ $ln(2).$
\end{center}
So (\ref{10}) is true for $\|f\|_{2,2d} = \|g\|_{2,2d} = 1$.  For generic $f,g\neq0$,
 let $\phi = \dfrac{f}{\|f\|_{2,2d}}$ and $\psi = \dfrac{g}{\|g\|_{2,2d}}$ so that $\|\phi\|_{2,2d} = \|\psi\|_{2,2d} = 1$ and $E(|\mathcal{G}_{\psi}(\phi)|^{2}) \geq$ $2$ $d$ $ln(2).$ . Since
$$\mathcal{G}_{\psi}(\phi) = \dfrac{\mathcal{G}_{g}(f)}{\|f\|_{2,2d}\|g\|_{2,2d}}$$
by combining Plancherel's formula (\ref{orth}) and Fubini's theorem we get
$$\begin{tabular}{lll}
$E(|\mathcal{G}_{\psi}(\phi)|^{2})$ & $=$ & $- \displaystyle\iint_{\mathbb{R}^{2d}\times\mathbb{R}^{2d}}\big(ln(|\mathcal{G}_{g}(f)(x,w)|^{2})-ln(\|f\|_{2,2d}^{2}\|g\|_{2,2d}^{2})\big) \dfrac{|\mathcal{G}_{g}(f)(x,w)|^{2}}{\|f\|_{2,2d}^{2}\|g\|_{2,2d}^{2}} d\mu_{4d}(x,w)$\\
& $=$ & $\dfrac{E(|\mathcal{G}_{g}(f)|^{2}) + ln(\|f\|^{2}_{2,2d}\|g\|_{2,2d}^{2})\|\mathcal{G}_{g}(f)\|^{2}_{2,4d}}{\|f\|^{2}_{2,2d} \|g\|^{2}_{2,2d}}$\\
& $=$ & $\dfrac{E(|\mathcal{G}_{g}(f)|^{2})}{\|f\|^{2}_{2,2d} \|g\|^{2}_{2,2d}} + ln(\|f\|^{2}_{2,2d} \|g\|^{2}_{2,2d})$\\
& $\geq$ & $2$ $d$ $ln(2).$
\end{tabular}$$
we deduce that
\begin{center}
$E(|\mathcal{G}_{g}(f)|^{2}) \geq \|f\|_{2,2d}^{2} \|g\|_{2,2d}^{2}$ ($2$ $d$ $ln(2)$ - $ln(\|f\|_{2,2d}^{2}\|g\|_{2,2d}^{2}).$
\end{center}
\end{proof}

\begin{cor}
Let $f , g \in L^{2}(\mathbb{R}^{2d})$ ; $g \neq 0$ , then
\begin{center}
$E(|A(f,g)|^{2}) \geq \|f\|_{2,2d}^{2} \|g\|_{2,2d}^{2}$  $\big(2\hspace{0.1 cm}d\hspace{0.1 cm} ln(2) - ln(\|f\|_{2,2d}^{2} \|g\|_{2,2d}^{2}) \big). $
\end{center}
\end{cor}
\begin{proof}
By using relation (\ref{p1}) we have
$$E(|A(f,g)|^{2}) = E(|\mathcal{G}_{g}(f)|^{2})$$
and consequently with relation (\ref{10}), we get
$$E(|A(f,g)|^{2}) \geq \|f\|_{2,2d}^{2} \|g\|_{2,2d}^{2} \big(2\hspace{0.1 cm}d\hspace{0.1 cm}ln(2) - ln(\|f\|_{2,2d}^{2} \|g\|_{2,2d}^{2})\big).$$
\end{proof}
\begin{cor}
Let $f , g \in L^{2}(\mathbb{R}^{2d},\mathbb{H})$ ; $g \neq 0$ , then,
$$E(|W(f,g)|^{2}) \geq \|f\|_{2,2d}^{2} \|g\|_{2,2d}^{2} \big(2\ d\ ln(2) - ln(2^{4d}\|f\|_{2,2d}^{2} \|g\|_{2,2d}^{2}) \big).$$
\end{cor}
\begin{proof}
By using  the relations (\ref{p2}) and (\ref{10}), we have
$$\begin{tabular}{lll}
$E(|W(f,g)|^{2})$ & $=$ & $-\displaystyle\iint_{\mathbb{R}^{2d}\times\mathbb{R}^{2d}} ln(|W(f,g)(x,w)|^{2}) |W(f,g)(x,w)|^{2} d\mu_{4d}(x,w)$\\
& $=$ & $-\displaystyle\int_{\mathbb{R}^{2d}\times\mathbb{R}^{2d}} ln(2^{4d}|\mathcal{G}_{\check{g}}(f)(2x,2w)|^{2}) 2^{4d} |\mathcal{G}_{\check{g}}(f)(2x,2w)|^{2} d\mu_{4d}(x,w)$\\
& $=$ & $-\displaystyle\int_{\mathbb{R}^{2d}\times\mathbb{R}^{2d}} ln(2^{4d}|\mathcal{G}_{\check{g}}(f)(y,\lambda)|^{2}) |\mathcal{G}_{\check{g}}(f)(y,\lambda)|^{2} d\mu_{4d}(y,\lambda)$\\
& $=$ & $-\displaystyle\int_{\mathbb{R}^{2d}\times\mathbb{R}^{2d}} \big[ln(2^{4d}) + ln(|\mathcal{G}_{\check{g}}(f)(y,\lambda)|^{2}) \big] |\mathcal{G}_{\check{g}}(f)(y,\lambda)|^{2} d\mu_{4d}(y,\lambda)$\\
& $=$ & $-ln(2^{4d})\|\mathcal{G}_{\check{g}}(f)\|_{2,2d}^{2} + E(|\mathcal{G}_{\check{g}}(f)|^{2})$\\
& $=$ & $-ln(2^{4d}) \|f\|_{2,2d}^{2} \|\check{g}\|_{2,2d}^{2} + E(|\mathcal{G}_{\check{g}}(f)|^{2})$\\
& $\geq$ & $-ln(2^{4d}) \|f\|_{2,2d}^{2} \|g\|_{2,2d}^{2} + \|f\|_{2,2d}^{2} \|g\|_{2,2d}^{2}(2\hspace{0.1 cm}d\hspace{0.1 cm}ln(2) - ln(\|f\|_{2,2d}^{2} \|g\|_{2,2d}^{2}))$\\
& $=$ & $ \|f\|_{2,2d}^{2} \|g\|_{2,2d}^{2}(2\hspace{0.1 cm}d\hspace{0.1 cm}ln(2) -ln(2^{4d})  - ln(\|f\|_{2,2d}^{2} \|g\|_{2,2d}^{2}))$\\
& $=$ & $ \|f\|_{2,2d}^{2} \|g\|_{2,2d}^{2}\big(2\hspace{0.1 cm}d\hspace{0.1 cm}ln(2) - ln(2^{4d} \|f\|_{2,2d}^{2} \|g\|_{2,2d}^{2})\big) .$
\end{tabular}$$
\end{proof}
In what follows we shall use The Beckner's uncertainty principle in terms of entropy to prove the Heisenberg uncertainty principle.
\subsection{The Heisenberg uncertainty principle  for $\mathcal{G}_{g}$}
\begin{thm} Let $p$ and $q$ be two positive real numbers. Then there exists a nonnegative constant $E_{p,q}$ such that for every window function $g$ and for every function $f \in L^{2}(\mathbb{R}^{2d},\mathbb{H})$ we have
\begin{eqnarray}
\bigg(\displaystyle\iint_{\mathbb{R}^{2d}\times\mathbb{R}^{2d}} |x|^{2p} |\mathcal{G}_{g}(f)(x,w)|^{2} d\mu_{4d}(x,w)\bigg)^{\frac{q}{p+q}} &  \bigg(\displaystyle\iint_{\mathbb{R}^{2d}\times\mathbb{R}^{2d}} |w|^{2q} |\mathcal{G}_{g}(f)(x,w)|^{2} d\mu_{4d}(x,w)\bigg)^{\frac{p}{p+q}}\nonumber \\
& \geq E_{p,q} \|f\|_{2,2d}^{2}\|g\|_{2,2d}^{2},
\end{eqnarray}
where
\begin{equation}\label{Epq}
E_{p,q} = \bigg[\bigg(\dfrac{p}{q}\bigg)^{\frac{q}{p+q}}+ \bigg(\dfrac{q}{p}\bigg)^{\frac{p}{p+q}}\bigg]^{-1}  e^{\frac{pq\bigg(2dln(2)+ln\bigg(\frac{2^{2d}pq\Gamma(d)^{2} }{\Gamma(\frac{d}{p})\Gamma(\frac{d}{q})}\bigg)\bigg)}{d(p+q)}-1}.
\end{equation}
\end{thm}
\begin{proof}
Assume that $\|f\|_{2,2d} = \|g\|_{2,2d} = 1 $ and let $A_{t,p,q}$ be the function defined on $\mathbb{R}^{2d}\times\mathbb{R}^{2d}$ by $A_{t,p,q}(x,w) = \dfrac{B_{p,q}}{t^{\frac{d}{p}+\frac{d}{q}}} e^{-\frac{|x|^{2p}+|w|^{2q}}{t}}$ where $B_{p,q} = \bigg(\dfrac{2^{2d}pq\Gamma(d)^{2} }{\Gamma(\frac{d}{p})\Gamma(\frac{d}{q})}\bigg)$.\\
$\bullet$ We see that
\begin{center}
 $ \displaystyle\iint_{\mathbb{R}^{2d}\times\mathbb{R}^{2d}} A_{t,p,q}(x,w) d\mu_{4d}(x,w) = 1$
\end{center}
in particular $d\sigma_{t,p,q}(x,w) = A_{t,p,q}(x,w) d\mu_{4d}(x,w)$ is a probability measure on $\mathbb{R}^{2d}\times\mathbb{R}^{2d}$.\\
$\bullet$ The function $\varphi(t) = t ln(t)$ is a convex function over $]0, +\infty[$.\\
$\bullet$ $H(x,w) = \dfrac{|\mathcal{G}_{g}(f)(x,w)|^{2}}{A_{t,p,q}}$ is a real-valued function, integrable with respect to the measure $d\sigma_{t,p,q}$ on $\mathbb{R}^{2d}\times\mathbb{R}^{2d}$,\\
hence according to Jensen inequality we get
\begin{center}
$\varphi\bigg(\displaystyle\iint_{\mathbb{R}^{2d}\times\mathbb{R}^{2d}} H(x,w)d\sigma_{t,p,q}(x,w)\bigg) \leq \displaystyle\iint_{\mathbb{R}^{2d}\times\mathbb{R}^{2d}} \varphi(H(x,w)) d\sigma_{t,p,q} (x,w)$
\end{center}
so
$$\hspace{-1 cm}\varphi\bigg(\displaystyle\iint_{\mathbb{R}^{2d}\times\mathbb{R}^{2d}} H(x,w)d\sigma_{t,p,q}(x,w)\bigg) = \varphi\bigg(\displaystyle\iint_{\mathbb{R}^{2d}\times\mathbb{R}^{2d}} \frac{|\mathcal{G}_{g}(f)(x,w)|^{2}}{A_{t,p,q}}  d\sigma_{t,p,q} (x,w)\bigg) = \varphi(\|\mathcal{G}_{g}(f)\|_{2,4d}^{2}) = \varphi(1) = 0$$
 and
 $$\displaystyle\iint_{\mathbb{R}^{2d}\times\mathbb{R}^{2d}} \varphi(H(x,w)) d\sigma_{t,p,q} (x,w) = \displaystyle\iint_{\mathbb{R}^{2d}\times\mathbb{R}^{2d}} \dfrac{|\mathcal{G}_{g}(f)(x,w)|^{2}}{A_{t,p,q}}ln\bigg(\dfrac{|\mathcal{G}_{g}(f)(x,w)|^{2}}{A_{t,p,q}}\bigg) d\sigma_{t,p,q} (x,w).$$
 Then
 \begin{center}
 $0 \leq \displaystyle\iint_{\mathbb{R}^{2d}\times\mathbb{R}^{2d}} \dfrac{|\mathcal{G}_{g}(f)(x,w)|^{2}}{A_{t,p,q}}ln\bigg(\dfrac{|\mathcal{G}_{g}(f)(x,w)|^{2}}{A_{t,p,q}}\bigg)d\sigma_{t,p,q} (x,w)$,
 \end{center}
which implies in terms of entropy that for every positive real number $t$,\\
$$\hspace{-1 cm}\begin{tabular}{lll}
$0$ & $ \leq $ & $\displaystyle\iint_{\mathbb{R}^{2d}\times\mathbb{R}^{2d}} \dfrac{|\mathcal{G}_{g}(f)(x,w)|^{2}}{A_{t,p,q}}ln\bigg(\dfrac{|\mathcal{G}_{g}(f)(x,w)|^{2}}{A_{t,p,q}}\bigg)d\sigma_{t,p,q} (x,w)$ \\
& $=$ & $ \displaystyle\iint_{\mathbb{R}^{2d}\times\mathbb{R}^{2d}} \dfrac{|\mathcal{G}_{g}(f)(x,w)|^{2}}{A_{t,p,q}}\big(ln(|\mathcal{G}_{g}(f)(x,w)|^{2}) - ln(A_{t,p,q})\big)d\sigma_{t,p,q} (x,w)$\\
& $=$ & $ \displaystyle\iint_{\mathbb{R}^{2d}\times\mathbb{R}^{2d}}|\mathcal{G}_{g}(f)(x,w)|^{2} ln(|\mathcal{G}_{g}(f)(x,w)|^{2}d\mu_{4d}(x,w) - \displaystyle\iint_{\mathbb{R}^{2d}\times\mathbb{R}^{2d}} |\mathcal{G}_{g}(f)(x,w)|^{2} ln(A_{t,p,q}(x,w)) d\mu_{4d}(x,w)$\\
& $=$ & $ - E(|\mathcal{G}_{g}(f)(x,w)|^{2}) - \displaystyle\iint_{\mathbb{R}^{2d}\times\mathbb{R}^{2d}} |\mathcal{G}_{g}(f)(x,w)|^{2} ln(A_{t,p,q}(x,w)) d\mu_{4d}(x,w)$\\
& $=$ & $ - E(|\mathcal{G}_{g}(f)(x,w)|^{2}) - \displaystyle\iint_{\mathbb{R}^{2d}\times\mathbb{R}^{2d}} \mathcal{G}_{g}(f)(x,w)|^{2} \bigg( ln(B_{p,q}) - \frac{|x|^{2p}+|w|^{2q}}{t} - ln(t^{\frac{d}{p}+\frac{d}{q}})\bigg) d\mu_{2d}(x,w)$\\
& $=$ & $ - E(|\mathcal{G}_{g}(f)(x,w)|^{2}) - ln(B_{p,q}) \|\mathcal{G}_{g}(f)(x,w)\|^{2}_{2,4d} + ln\big(t^{\frac{d}{p}+\frac{d}{q}}\big)\|\mathcal{G}_{g}(f)(x,w)\|^{2}_{2,4d} $\\
& $+$ & $ \dfrac{1}{t} \displaystyle\iint_{\mathbb{R}^{2d}\times\mathbb{R}^{2d}} (|x|^{2p}+|w|^{2q})|\mathcal{G}_{g}(f)(x,w)|^{2} d\mu_{4d}(x,w)$\\
& $=$ & $ - E(|\mathcal{G}_{g}(f)(x,w)|^{2}) - ln(B_{p,q})  + ln\big(t^{\frac{d}{p}+\frac{d}{q}}\big)  + \dfrac{1}{t} \displaystyle\iint_{\mathbb{R}^{2d}\times\mathbb{R}^{2d}} (|x|^{2p}+|w|^{2q})|\mathcal{G}_{g}(f)(x,w)|^{2} d\mu_{4d}(x,w)$.
\end{tabular}$$
Therefore
$$E(|\mathcal{G}_{g}(f)(x,w)|^{2}) + ln(B_{p,q})  \leq ln\big(t^{\frac{d}{p}+\frac{d}{q}}\big) + \dfrac{1}{t} \displaystyle\iint_{\mathbb{R}^{2d}\times\mathbb{R}^{2d}} (|x|^{2p}+|w|^{2q})|\mathcal{G}_{g}(f)(x,w)|^{2} d\mu_{4d}(x,w).$$
By (\ref{10}) we get\\

$\|f\|_{2,2d}^{2} \|g\|_{2,2d}^{2}(2\hspace{0.1 cm}d\hspace{0.1 cm}ln(2) - ln(\|f\|_{2,2d}^{2} \|g\|_{2,2d}^{2})) + ln(B_{p,q})$
\begin{center}
$\leq ln\big(t^{\frac{d}{p}+\frac{d}{q}}\big)+ \dfrac{1}{t} \displaystyle\iint_{\mathbb{R}^{2d}\times\mathbb{R}^{2d}} (|x|^{2p}+|w|^{2q})|\mathcal{G}_{g}(f)(x,w)|^{2} d\mu_{4d}(x,w).$
\end{center}
Knowing that $\|f\|_{2,2d} =  \|g\|_{2,2d} = 1$, we deduce that
\begin{center}
$t \bigg(2\hspace{0.1 cm}d\hspace{0.1 cm}ln(2) + ln(B_{p,q}) - ln\big(t^{\frac{d}{p}+\frac{d}{q}}\big)\bigg) \leq  \displaystyle\iint_{\mathbb{R}^{2d}\times\mathbb{R}^{2d}} (|x|^{2p}+|w|^{2q})|\mathcal{G}_{g}(f)(x,w)|^{2} d\mu_{4d}(x,w).$
\end{center}
However the expression $t \bigg(2\hspace{0.1 cm}d\hspace{0.1 cm}ln(2) + ln(B_{p,q}) - ln\big(t^{\frac{d}{p}+\frac{d}{q}}\big)\bigg)$ attains its upper bound if
$$ 0  = \dfrac{d}{dt}\bigg(t \big(2\hspace{0.1 cm}d\hspace{0.1 cm}ln(2) + ln(B_{p,q}) - ln\big(t^{\frac{d}{p}+\frac{d}{q}}\big)\big)\bigg) =  2dln(2) + ln(B_{p,q}) -  ln\big(t^{\frac{d}{p}+\frac{d}{q}}\big) - \bigg(\frac{d}{p}+\frac{d}{q}\bigg) $$
at the point
  $t = e^{\frac{pq(2dln(2)+ln(B_{p,q}))}{d(p+q)}-1}$\\
 which implies that
\begin{center}
$D_{p,q} \leq  \displaystyle\iint_{\mathbb{R}^{2d}\times\mathbb{R}^{2d}} (|x|^{2p}+|w|^{2q})|\mathcal{G}_{g}(f)(x,w)|^{2} d\mu_{4d}(x,w) $
\end{center}
 where $D_{p,q} = \frac{d(p+q)}{pq}e^{\frac{pq(2dln(2)+ln(B_{p,q}))}{d(p+q)}-1}.$\\
Suppose now that $f\neq 0$, $\phi=\dfrac{f}{\|f\|_{2,2d}}$ and $\psi=\dfrac{g}{\|g\|_{2,2d}}$, so that $\|\phi\|_{2,2d} = \|\psi\|_{2,2d} = 1$.\\
By the previous calculations we have
\begin{center}
$D_{p,q}  \leq  \displaystyle\iint_{\mathbb{R}^{2d}\times\mathbb{R}^{2d}} (|x|^{2p}+|w|^{2q})|\mathcal{G}_{\psi}(\phi)(x,w)|^{2} d\mu_{4d}(x,w) $.
\end{center}
Using the relation $|\mathcal{G}_{\psi}(\phi)(x,w)| = \dfrac{|\mathcal{G}_{g}(f)(x,w)|}{ \|f\|_{2,2d} \|g\|_{2,2d}}$, we get
\begin{center}
$D_{p,q}\hspace{0.1 cm} \|f\|^{2}_{2,2d} \|g\|^{2}_{2,2d}  \leq  \displaystyle\iint_{\mathbb{R}^{2d}\times\mathbb{R}^{2d}} (|x|^{2p}+|w|^{2q})|\mathcal{G}_{g}(f)(x,w)|^{2} d\mu_{4d}(x,w) .$
\end{center}
Now for every positive real number $\lambda$ the dilates $f_{\lambda}$ and $g_{\lambda}$ belongs to $L^{2}(\mathbb{R}^{2d},\mathbb{H})$, we have
\begin{center}
$D_{p,q}\hspace{0.1 cm} \|f_{\lambda}\|^{2}_{2,2d} \|g_{\lambda}\|^{2}_{2,2d}  \leq  \displaystyle\iint_{\mathbb{R}^{2d}\times\mathbb{R}^{2d}} (|x|^{2p}+|w|^{2q})|\mathcal{G}_{g_{\lambda}}(f_{\lambda})(x,w)|^{2} d\mu_{4d}(x,w)$
\end{center}
 then by relation (\ref{dilate}), we have\\
 $\displaystyle\iint_{\mathbb{R}^{2d}\times\mathbb{R}^{2d}} (|x|^{2p}+|w|^{2q})|\mathcal{G}_{g_{\lambda}}(f_{\lambda})(x,w)|^{2} d\mu_{4d}(x,w) $
 \begin{eqnarray*}
  & = & \displaystyle\iint_{\mathbb{R}^{2d}\times\mathbb{R}^{2d}} (|x|^{2p}+|w|^{2q})\bigg|\dfrac{\mathcal{G}_{g}(f)(\lambda x,\frac{w}{\lambda})}{\lambda^{2d}}\bigg|^{2} d\mu_{4d}(x,w) \\
 & = & \dfrac{1}{\lambda^{4d}}\displaystyle\iint_{\mathbb{R}^{2d}\times\mathbb{R}^{2d}}\bigg(\bigg|\frac{x}{\lambda}\bigg|^{2p} + |\lambda w|^{2q}\bigg)|\mathcal{G}_{g}(f)( x,w)|^{2} d\mu_{4d}(x,w)
 \end{eqnarray*}
 and
 \begin{center}
 $D_{p,q}\hspace{0.1 cm} \|f_{\lambda}\|^{2}_{2,2d} \|g_{\lambda}\|^{2}_{2,2d} = \dfrac{1}{\lambda^{4d}} D_{p,q} \hspace{0.1 cm} \|f\|^{2}_{2,2d} \|g\|^{2}_{2,2d}$.
 \end{center}
 Therefore for every positive real number $\lambda$
\begin{center}
$D_{p,q}\hspace{0.1 cm} \|f\|^{2}_{2,2d} \|g\|^{2}_{2,2d}  \leq  Q_{\lambda}$,
\end{center}
where
\begin{equation*}
Q_{\lambda} = \displaystyle\iint_{\mathbb{R}^{2d}\times\mathbb{R}^{2d}} (\lambda^{-2p}|x|^{2p} + \lambda^{2q}|w|^{2q})|\mathcal{G}_{g}(f)(x,w)|^{2} d\mu_{4d}(x,w)
\end{equation*}
in particular, the inequality holds at the critical point $\lambda$ where $\dfrac{d}{d\lambda} Q_{\lambda} = 0$
then for
\begin{equation*}
\lambda =\bigg(\dfrac{p}{q}\displaystyle\iint_{\mathbb{R}^{2d}\times\mathbb{R}^{2d}} |x|^{2p} |\mathcal{G}_{g}(f)(x,w)|^{2} d\mu_{4d}(x,w)\bigg)^{\frac{1}{2(p+q)}} \bigg(\displaystyle\iint_{\mathbb{R}^{2d}\times\mathbb{R}^{2d}} |w|^{2q} |\mathcal{G}_{g}(f)(x,w)|^{2} d\mu_{4d}(x,w)\bigg)^{\frac{-1}{2(p+q)}}
\end{equation*}
we have that\\
$D_{p,q}\|f\|^{2}_{2,2d} \|g\|^{2}_{2,2d} \leq \bigg[\bigg(\dfrac{p}{q}\bigg)^{\frac{q}{p+q}}+ \bigg(\dfrac{q}{p}\bigg)^{\frac{p}{p+q}}\bigg]$
\begin{eqnarray*}
  \bigg(\displaystyle\iint_{\mathbb{R}^{2d}\times\mathbb{R}^{2d}} |x|^{2p} |\mathcal{G}_{g}(f)(x,w)|^{2} d\mu_{4d}(x,w)\bigg)^{\frac{q}{p+q}}\bigg(\displaystyle\iint_{\mathbb{R}^{2d}\times\mathbb{R}^{2d}} |w|^{2q} |\mathcal{G}_{g}(f)(x,w)|^{2} d\mu_{4d}(x,w)\bigg)^{\frac{p}{p+q}}.
\end{eqnarray*}
Therefore
\begin{center}
$E_{p,q} \|f\|^{2}_{2,2d} \|g\|^{2}_{2,2d} \leq \bigg(\displaystyle\iint_{\mathbb{R}^{2d}\times\mathbb{R}^{2d}} |x|^{2p} |\mathcal{G}_{g}(f)(x,w)|^{2} d\mu_{4d}(x,w)\bigg)^{\frac{q}{p+q}}\bigg(\displaystyle\iint_{\mathbb{R}^{2d}\times\mathbb{R}^{2d}} |w|^{2q} |\mathcal{G}_{g}(f)(x,w)|^{2} d\mu_{4d}(x,w)\bigg)^{\frac{p}{p+q}}$
\end{center}
where
\begin{center}
$E_{p,q} = \bigg[\bigg(\dfrac{p}{q}\bigg)^{\frac{q}{p+q}}+ \bigg(\dfrac{q}{p}\bigg)^{\frac{p}{p+q}}\bigg]^{-1}  e^{\frac{pq\bigg(2dln(2)+ln\bigg(\frac{2^{2d}pq\Gamma(d)^{2} }{\Gamma(\frac{d}{p})\Gamma(\frac{d}{q})}\bigg)\bigg)}{d(p+q)}-1}$ .
\end{center}
In the particular case when $p = q = 1$ we get
\begin{center}
$\||x|  \mathcal{G}_{g}(f)(x,w)\|_{2,4d}\| |w| \mathcal{G}_{g}(f)(x,w)\|_{2,4d} \geq \dfrac{2}{e} \|f\|_{2,2d}^{2} \|g\|_{2,2d}^{2}$
\end{center}
\end{proof}
\begin{cor}\label{hesss4}Let $p$ and $q$ be two positive real numbers. Then there exists a nonnegative constant $E_{p,q}$ such that for every window function $g$ and for every function $f \in L^{2}(\mathbb{R}^{2d},\mathbb{H})$ we have
\begin{equation}\label{hes4}
\||(x,w)|^{p}\mathcal{G}_{g}(f)\|_{2,4d}^{\frac{q}{p+q}} \||(x,w)|^{q}\mathcal{G}_{g}(f)\|_{2,4d}^{\frac{p}{p+q}} \geq \sqrt{E_{p,q}} \|f\|_{2,2d}  \|g\|_{2,2d},
\end{equation}
where $E_{p,q}$ is given by (\ref{Epq}).
\end{cor}
\begin{cor} Let $p$ and $q$ be two positive real numbers. Then there exists a nonnegative constant $E_{p,q}$ such that for every window function $g$ and for every function $f \in L^{2}(\mathbb{R}^{2d},\mathbb{H})$ we have
\begin{eqnarray*}
\||(x,w)|^{p}\mathcal{A}(f,g)\|_{2,4d}^{\frac{q}{p+q}} \||(x,w)|^{q}\mathcal{A}(f,g)\|_{2,4d}^{\frac{p}{p+q}} \geq \sqrt{E_{p,q}} \|f\|_{2,2d}  \|g\|_{2,2d},
\end{eqnarray*}
where $E_{p,q}$ is given by (\ref{Epq}).
\end{cor}
\begin{cor} Let $p$ and $q$ be two positive real numbers. Then there exists a nonnegative constant $E_{p,q}$ such that for every window function $g$ and for every function $f \in L^{2}(\mathbb{R}^{2d},\mathbb{H})$ we have
\begin{eqnarray*}
\||(x,w)|^{p}\mathcal{W}(f,g)\|_{2,4d}^{\frac{q}{p+q}} \||(x,w)|^{q}\mathcal{W}(f,g)\|_{2,4d}^{\frac{p}{p+q}} \geq 4^{-\frac{pq}{p+q}}\sqrt{E_{p,q}} \|f\|_{2,2d}  \|g\|_{2,2d},\\
\end{eqnarray*}
where $E_{p,q}$ is given by (\ref{Epq}).
\end{cor}
\begin{proof}
By using the corollary (\ref{p1}), we have\\
$\bigg(\displaystyle\iint_{\mathbb{R}^{2d}\times\mathbb{R}^{2d}} |(x,w)|^{p} |\mathcal{W}(f,g)(x,w)|^{2} d\mu_{4d}(x,w)\bigg)^{\frac{q}{p+q}} \bigg(\displaystyle\iint_{\mathbb{R}^{2d}\times\mathbb{R}^{2d}} |(x,w)|^{q} |\mathcal{W}(f,g)(x,w)|^{2} d\mu_{4d}(x,w)\bigg)^{\frac{p}{p+q}}$\\
$= \bigg(\displaystyle\int_{\mathbb{R}^{4d}} |(x,w)|^{p} |2^{2d}\mathcal{G}_{\check{g}}(f)(2x,2w)|^{2} d\mu_{4d}(x,w)\bigg)^{\frac{q}{p+q}} \bigg(\displaystyle\int_{\mathbb{R}^{4d}} |(x,w)|^{q} |2^{2d}\mathcal{G}_{\check{g}}(f)(2x,2w)|^{2} d\mu_{4d}(x,w)\bigg)^{\frac{p}{p+q}}$\\
$= \bigg(\dfrac{1}{2}\bigg)^{\frac{pq}{p+q}} \bigg(\displaystyle\int_{\mathbb{R}^{4d}} |(y,\sigma)|^{p} |\mathcal{G}_{\check{g}}(f)(y,\sigma)|^{2} d\mu_{4d}(y,\sigma)\bigg)^{\frac{q}{p+q}} \bigg(\dfrac{1}{2}\bigg)^{\frac{pq}{p+q}}\bigg(\displaystyle\int_{\mathbb{R}^{4d}} |(y,\sigma)|^{q} |\mathcal{G}_{\check{g}}(f)(y,\sigma)|^{2} d\mu_{4d}(y,\sigma)\bigg)^{\frac{p}{p+q}}$\\
$= \bigg(\dfrac{1}{4}\bigg)^{\frac{pq}{p+q}} \bigg(\displaystyle\int_{\mathbb{R}^{4d}} |(y,\sigma)|^{p} |\mathcal{G}_{\check{g}}(f)(y,\sigma)|^{2} d\mu_{4d}(y,\sigma)\bigg)^{\frac{q}{p+q}} \bigg(\displaystyle\int_{\mathbb{R}^{4d}} |(y,\sigma)|^{q} |\mathcal{G}_{\check{g}}(f)(y,\sigma)|^{2} d\mu_{4d}(y,\sigma)\bigg)^{\frac{p}{p+q}}$\\
$\geq  \bigg(\dfrac{1}{4}\bigg)^{\frac{pq}{p+q}} E_{p,q} \|f\|_{2,2d}^{2}\|g\|_{2,2d}^{2} $
\end{proof}
\subsection{Local Price's inequality}
\begin{thm}
\mbox{}\\
Let $\varepsilon$ , $p$ be two positive real numbers such that $0< \varepsilon < 2d$ and $p \geq 1$ then there is a nonnegative constant $M_{\varepsilon,\mu}$ such that for every  quaternion windowed function  $g$ , for every function $f \in L^{2}(\mathbb{R}^{2d},\mathbb{H})$ and for every finite measurable subset $\Sigma$ of $\mathbb{R}^{2d}\times\mathbb{R}^{2d}$ , we have
\begin{equation}\label{2}
\|\chi_{\Sigma}\mathcal{G}_{g}(f)\|_{p,4d}^{p(p+1)} \leq M_{\varepsilon,p} \times (\mu_{4d}(\Sigma)) \times  \| |(x,w)|^{\varepsilon} \mathcal{G}_{g}(f) \|_{2 ,4d}^{\frac{4pd}{(2d+\varepsilon)}}(\|f\|_{2,2d}\|g\|_{2,2d})^{p(p-\frac{2d-\varepsilon}{2d+\varepsilon})}
\end{equation}
where
$M_{\varepsilon,p} = \bigg(\dfrac{2d+\varepsilon}{2^{\frac{\varepsilon(2d+2p+2)}{(2d+\varepsilon)(P+1)}}\varepsilon^{\frac{2\varepsilon}{2d+\varepsilon}}\Gamma(2d)^{\frac{\varepsilon}{(2d+\varepsilon)(p+1)}}(2d-\varepsilon)^{\frac{2d-\varepsilon}{2d+\varepsilon}+\frac{\varepsilon}{(2d+\varepsilon)(p+1)}}}\bigg)^{p(p+1)}$
\end{thm}
\textbf{Proof}\\
Without loss of generality we can assume that $\|f\|_{2,2d} = \|g\|_{2,2d} = 1$ , then for every positive real number $s$ , we have
$$\|\chi_{\Sigma}\mathcal{G}_{g}(f)\|_{p,4d} \leq \|\mathcal{G}_{g}(f)\chi_{B_{s}\cap\Sigma}\|_{p,4d} + \|\mathcal{G}_{g}(f)\chi_{B_{s}^{c}\cap\Sigma}\|_{p,4d}$$
where $B_{s}$ denotes the ball of $\mathbb{R}^{2d}\times\mathbb{R}^{2d}$ of radius $s$ . However , by Holder's inequality and relation (\ref{1}) we get for every $0< \varepsilon < 2d$
$$\begin{tabular}{lll}
$\|\mathcal{G}_{g}(f) \chi_{B_{s}\cap\Sigma}\|_{p,4d} $ & $=$ & $\bigg( \displaystyle\iint_{\mathbb{R}^{2d}\times\mathbb{R}^{2d}} |\mathcal{G}_{g}(f)(x,w)|^{p} \chi_{B_{s}}(x,w)\chi_{\Sigma}(x,w) d\mu_{4d}(x,w)\bigg)^{\frac{1}{p}}$\\
& $=$ & $\bigg(\displaystyle\iint_{\mathbb{R}^{2d}\times\mathbb{R}^{2d}} |\mathcal{G}_{g}(f)(x,w)|^{\frac{p^{2}}{p+1}} |\mathcal{G}_{g}(f)(x,w)|^{\frac{p}{p+1}} \chi_{B_{s}}(x,w)\chi_{\Sigma}(x,w) d\mu_{4d}(x,w)\bigg)^{\frac{1}{p}}$\\
& $\leq$ & $\|\mathcal{G}_{g}(f)\|_{\infty,4d}^{\frac{p}{p+1}}\bigg( \displaystyle\iint_{\mathbb{R}^{2d}\times\mathbb{R}^{2d}} |\mathcal{G}_{g}(f)(x,w)|^{\frac{p}{p+1}} \chi_{B_{s}}(x,w)\chi_{\Sigma}(x,w) d\mu_{4d}(x,w)\bigg)^{\frac{1}{p}}$\\
& $\leq$ & $\bigg( \displaystyle\iint_{\mathbb{R}^{2d}\times\mathbb{R}^{2d}} |\mathcal{G}_{g}(f)(x,w)|^{\frac{p}{p+1}} \chi_{B_{s}}(x,w)\chi_{\Sigma}(x,w) d\mu_{4d}(x,w)\bigg)^{\frac{1}{p}}$\\
& $\leq$ & $\bigg( \displaystyle\iint_{\mathbb{R}^{2d}\times\mathbb{R}^{2d}} |\mathcal{G}_{g}(f)(x,w)|\chi_{B_{s}}(x,w) d\mu_{4d}(x,w)\bigg)^{\frac{p}{p(p+1)}}$\\
& $\times$ & $\bigg(\displaystyle\iint_{\mathbb{R}^{2d}\times\mathbb{R}^{2d}}\chi_{\Sigma}(x,w) d\mu_{4d}(x,w)\bigg)^{\frac{1}{p(p+1)}}$\\
& $=$ & $\mu_{4d}(\Sigma)^{\frac{1}{p(p+1)}}\|\mathcal{G}_{g}(f)\chi_{B_{s}}\|_{1,4d}^{\frac{1}{p+1}}$\\
& $\leq $ & $\mu_{4d}(\Sigma)^{\frac{1}{p(p+1)}}\||(x,w)|^{\varepsilon} G_{g}(f) \|_{2,4d}^{\frac{1}{p+1}}$ $\| |(x,w)|^{-\varepsilon}\chi_{B_{s}}\|_{2,4d}^{\frac{1}{p+1}}$.
\end{tabular} $$
On the other hand
$$\| |(x,w)|^{-\varepsilon}\chi_{B_{s}}\|_{2,4d}^{\frac{1}{p+1}} = \bigg(\displaystyle\iint_{\mathbb{R}^{2d}\times\mathbb{R}^{2d}} |(x,w)|^{-2\varepsilon} \chi_{B_{s}} d\mu_{4d}(x,w)\bigg)^{\frac{1}{2(p+1)}} = \dfrac{s^{\frac{2d-\varepsilon}{p+1}}}{(2^{2d}\Gamma(2d)(2d-\varepsilon))^{\frac{1}{2(p+1)}}}$$
so
$$\|\mathcal{G}_{g}(f)\chi_{B_{s}\cap\Sigma}\|_{p,4d} \leq \dfrac{\mu_{4d}(\Sigma)^{\frac{1}{p(p+1)}}s^{\frac{2d-\varepsilon}{p+1}}}{(2^{2d}\Gamma(2d)(2d-\varepsilon))^{\frac{1}{2(p+1)}}} \| |(x,w)|^{\varepsilon}\mathcal{G}_{g}(f)\|_{2,4d}^{\frac{1}{p+1}}.$$
On the other hand, and again by Holder's inequality and relation (\ref{1}), we deduce that
$$\begin{tabular}{lll}
$\|\mathcal{G}_{g}(f) \chi_{B_{s}^{c}\cap\Sigma}\|_{p,4d}$ & $=$ &
$\bigg(\displaystyle\iint_{\mathbb{R}^{2d}\times\mathbb{R}^{2d}}
|\mathcal{G}_{g}(f)(x,w)|^{p}\chi_{B_{s}^{c}}(x,w)\chi_{\Sigma}(x,w) d\mu_{4d}(x,w)\bigg)^{\frac{1}{p}}$\\
& $=$ &
$\bigg(\displaystyle\iint_{\mathbb{R}^{2d}\times\mathbb{R}^{2d}}
|\mathcal{G}_{g}(f)(x,w)|^{\frac{2p}{p+1}}|\mathcal{G}_{g}(f)(x,w)|^{\frac{p(p-1)}{p+1}}\chi_{B_{s}^{c}}(x,w)\chi_{\Sigma}(x,w) d\mu_{4d}(x,w)\bigg)^{\frac{1}{p}}$\\
& $=$ & $\|G_{g}(f)\|_{\infty,4d}^{\frac{p-1}{p+1}}\bigg(\displaystyle\iint_{\mathbb{R}^{2d}\times\mathbb{R}^{2d}} |\mathcal{G}_{g}(f)(x,w)|^{\frac{2p}{p+1}} \chi_{B_{s}^{c}}(x,w)\chi_{\Sigma}(x,w) d\mu_{4d}(x,w)\bigg)^{\frac{1}{p}}$\\
& $\leq$ & $\bigg(\displaystyle\iint_{\mathbb{R}^{2d}\times\mathbb{R}^{2d}} |\mathcal{G}_{g}(f)(x,w)|^{\frac{2p}{p+1}} \chi_{B_{s}^{c}}(x,w)\chi_{\Sigma}(x,w) d\mu_{4d}(x,w)\bigg)^{\frac{1}{p}}$\\
& $=$ & $\bigg(\displaystyle\iint_{\mathbb{R}^{2d}\times\mathbb{R}^{2d}} (|\mathcal{G}_{g}(f)(x,w)|^{2}\chi_{B_{s}^{c}}(x,w))^{\frac{p}{p+1}} (\chi_{\Sigma}(x,w))^{\frac{1}{p+1}} d\mu_{4d}(x,w)\bigg)^{\frac{1}{p}}$\\
& $\leq$ & $\mu_{4d}(\Sigma)^{\frac{1}{p(p+1)}}\bigg(\displaystyle\iint_{\mathbb{R}^{2d}\times\mathbb{R}^{2d}} |\mathcal{G}_{g}(f)(x,w)|^{2}\chi_{B_{s}^{c}}(x,w) d\mu_{4d}(x,w)\bigg)^{\frac{1}{p+1}}$\\
& $=$ & $\mu_{4d}(\Sigma)^{\frac{1}{p(p+1)}} \| |(x,w)|^{-2\varepsilon}\chi_{B_{s}^{c}}\|_{\infty,4d}^{\frac{1}{p+1}} \| |(x,w)|^{\varepsilon} \mathcal{G}_{g}(f)\|_{2;4d}^{\frac{2}{p+1}}$\\
& $\leq$ & $\mu_{4d}(\Sigma)^{\frac{1}{p(p+1)}} s^{\frac{-2\varepsilon}{p+1}} \| |(x,w)|^{\varepsilon} \mathcal{G}_{g}(f)\|_{2,4d}^{\frac{2}{p+1}} .$
\end{tabular}$$
Hence,
$$\bigg(\displaystyle\int\displaystyle\int_{\Sigma} |\mathcal{G}_{g}(f)(x,w)|^{p} d\mu_{4d}(x,w)\bigg)^{\frac{1}{p}} \leq $$
$(\mu_{4d}(\Sigma))^{\frac{1}{p(p+1)}} \| |(x,w)|^{\varepsilon} \mathcal{G}_{g}(f) \|_{2 , 4d}^{\frac{1}{p+1}} \times \bigg(\dfrac{s^{\frac{2d-\varepsilon}{p+1}}}{(2^{2d}\Gamma(2d)(2d-\varepsilon))^{\frac{1}{2(p+1)}}} + \| |(x,w)|^{\varepsilon} \mathcal{G}_{g}(f)\|_{2, 4d}^{\frac{1}{p+1}} s^{-\frac{2\varepsilon}{p+1}}\bigg)$\\
In particular the inequality holds for
$$s_{0} = \bigg(\dfrac{2 \varepsilon \|(x,w)|^{\varepsilon} \mathcal{G}_{g}(f)\|_{2,4d}^{\frac{1}{p+1}}(2^{2d}\Gamma(2d)(2d-\varepsilon))^{\frac{1}{2(p+1)}}}{2d-\varepsilon}\bigg)^{\frac{p+1}{2d+\varepsilon}}$$
and therefore
\begin{center}
$\bigg(\displaystyle\int\displaystyle\int_{\Sigma} |\mathcal{G}_{g}(f)(x,w)|^{p} d\mu_{4d}(x,w)\bigg)^{\frac{1}{p}} \leq
(\mu_{4d}(\Sigma))^{\frac{1}{p(p+1)}} \| |(x,w)|^{\varepsilon} \mathcal{G}_{g}(f) \|_{2 , 4d}^{\frac{4d}{(2d+\varepsilon)(p+1)}} \times \bigg(\dfrac{2d+\varepsilon}{2^{\frac{\varepsilon(2d+2p+2)}{(2d+\varepsilon)(P+1)}}\varepsilon^{\frac{2\varepsilon}{2d+\varepsilon}}\Gamma(2d)^{\frac{\varepsilon}{(2d+\varepsilon)(p+1)}}(2d-\varepsilon)^{\frac{2d-\varepsilon}{2d+\varepsilon}+\frac{\varepsilon}{(2d+\varepsilon)(p+1)}}}\bigg)$
\end{center}
If $f = 0$, then (\ref{2}) is true. suppose that $f \neq 0$ , $\phi = \dfrac{f}{\|f\|_{2,2d}}$ and $\psi = \dfrac{g}{\|g\|_{2,2d}}$ ,we have $\|\phi\|_{2,2d} = \|\psi\|_{2,2d} = 1$ and
\begin{center}
$\bigg(\displaystyle\int\displaystyle\int_{\Sigma} |\mathcal{G}_{\psi}(\phi)(x,w)|^{p} d\mu_{4d}(x,w)\bigg)^{\frac{1}{p}} \leq
(\mu_{4d}(\Sigma))^{\frac{1}{p(p+1)}} \| |(x,w)|^{\varepsilon} \mathcal{G}_{\psi}(\phi) \|_{2 , \mathbb{R}^{2d}\times\mathbb{R}^{2d}}^{\frac{4d}{(2d+\varepsilon)(p+1)}} \times \bigg(\dfrac{2d+\varepsilon}{2^{\frac{\varepsilon(2d+2p+2)}{(2d+\varepsilon)(P+1)}}\varepsilon^{\frac{2\varepsilon}{2d+\varepsilon}}\Gamma(2d)^{\frac{\varepsilon}{(2d+\varepsilon)(p+1)}}(2d-\varepsilon)^{\frac{2d-\varepsilon}{2d+\varepsilon}+\frac{\varepsilon}{(2d+\varepsilon)(p+1)}}}\bigg)$ .
\end{center}
We have
$$\mathcal{G}_{\psi}(\phi) = \dfrac{\mathcal{G}_{g}(f)}{\|f\|_{2,2d}\|g\|_{2,2d}},$$
then
$$\begin{tabular}{lll}
$\bigg(\displaystyle\iint_{\Sigma} |\mathcal{G}_{\psi}(\phi)(x,w)|^{p} d\mu_{4d}(x,w)\bigg)^{\frac{1}{p}} $ & $\leq$ & $\Delta (\mu_{4d}(\Sigma))^{\frac{1}{p(p+1)}} \| |(x,w)|^{\varepsilon} \mathcal{G}_{\psi}(\phi) \|_{2 , 4d}^{\frac{4d}{(2d+\varepsilon)(p+1)}} $\\
& $\leq$ & $\Delta \dfrac{(\mu_{4d}(\Sigma))^{\frac{1}{p(p+1)}} \| |(x,w)|^{\varepsilon} \mathcal{G}_{g}(f) \|_{2 , 4d}^{\frac{4d}{(2d+\varepsilon)(p+1)}}}{(\|f\|_{2,2d}\|g\|_{2,2d})^{\frac{4d}{(2d+\varepsilon)(p+1)}}}$ ;
\end{tabular}$$
where  $\Delta =  \dfrac{2d+\varepsilon}{2^{\frac{\varepsilon(2d+2p+2)}{(2d+\varepsilon)(P+1)}}\varepsilon^{\frac{2\varepsilon}{2d+\varepsilon}}\Gamma(2d)^{\frac{\varepsilon}{(2d+\varepsilon)(p+1)}}(2d-\varepsilon)^{\frac{2d-\varepsilon}{2d+\varepsilon}+\frac{\varepsilon}{(2d+\varepsilon)(p+1)}}}$ .\\
Hence
$$\bigg(\displaystyle\int\displaystyle\int_{\Sigma} |\mathcal{G}_{g}(f)(x,w)|^{p} d\mu_{4d}(x,w)\bigg)^{\frac{1}{p}} \leq \Delta (\mu_{4d}(\Sigma))^{\frac{1}{p(p+1)}} \| |(x,w)|^{\varepsilon} \mathcal{G}_{g}(f) \|_{2 , 4d}^{\frac{4d}{(2d+\varepsilon)(p+1)}} (\|f\|_{2,2d}\|g\|_{2,2d})^{1-\frac{4d}{(2d+\varepsilon)(p+1)}}$$
and consequently
$$\|\chi_{\Sigma}\mathcal{G}_{g}(f)\|_{p,4d}^{p(p+1)} \leq \Delta^{p(p+1)} \times (\mu_{4d}(\Sigma)) \times  \| |(x,w)|^{\varepsilon} \mathcal{G}_{g}(f) \|_{2 , 4d}^{\frac{4pd}{(2d+\varepsilon)}}(\|f\|_{2,2d}\|g\|_{2,2d})^{p(p-\frac{2d-\varepsilon}{2d+\varepsilon})}$$
we note
$$M_{\varepsilon,p} = \Delta^{p(p+1)} = \bigg(\dfrac{2d+\varepsilon}{2^{\frac{\varepsilon(2d+2p+2)}{(2d+\varepsilon)(P+1)}}\varepsilon^{\frac{2\varepsilon}{2d+\varepsilon}}\Gamma(2d)^{\frac{\varepsilon}{(2d+\varepsilon)(p+1)}}(2d-\varepsilon)^{\frac{2d-\varepsilon}{2d+\varepsilon}+\frac{\varepsilon}{(2d+\varepsilon)(p+1)}}}\bigg)^{p(p+1)}$$
finally
$$\|\chi_{\Sigma}\mathcal{G}_{g}(f)\|_{p,4d}^{p(p+1)} \leq M_{\varepsilon,p} \times (\mu_{4d}(\Sigma)) \times  \| |(x,w)|^{\varepsilon} \mathcal{G}_{g}(f) \|_{2 , 4d}^{\frac{4pd}{(2d+\varepsilon)}}(\|f\|_{2,2d}\|g\|_{2,2d})^{p(p-\frac{2d-\varepsilon}{2d+\varepsilon})}$$
the proof is complete .
\begin{cor}
\mbox{}\\
Let $\varepsilon$ , p be two positive real numbers such that $0<\varepsilon<2d$ and $p \geq 1$, then there is a non negative constant $M_{\varepsilon,p}$ such that for every function $f \in L^{2}(\mathbb{R}^{2d},\mathbb{H})$ and for every finite measurable set $\Sigma$ of $\mathbb{R}^{2d}\times\mathbb{R}^{2d}$, we have
\begin{equation}\label{3}
\|\chi_{\Sigma}A(f,g)\|_{p,4d}^{p(p+1)} \leq M_{\varepsilon,p} \times (\mu_{4d}(\Sigma)) \times  \| |(x,w)|^{\varepsilon} A(f,g) \|_{2 , 4d}^{\frac{4pd}{(2d+\varepsilon)}}(\|f\|_{2,2d}\|g\|_{2,2d})^{p(p-\frac{2d-\varepsilon}{2d+\varepsilon})}.
\end{equation}
\end{cor}
\textbf{Proof}\\
We have from the relation (\ref{p1}), $|A(f,g)(x,w)| = |\mathcal{G}_{g}(f)(x,w)|$, then
$$\|\chi_{\Sigma}A(f,g)\|_{p,4d} = \|\chi_{\Sigma} \mathcal{G}_{g}(f)\|_{p,4d}$$
and
$$\| |(x,w)|^{\varepsilon} A(f,g)\|_{p,4d} = \||(x,w)|^{\varepsilon} \mathcal{G}_{g}(f)\|_{p,4d} $$
then, with relation (\ref{2}), we have
$$\|\chi_{\Sigma}A(f,g)\|_{p,4d}^{p(p+1)} \leq M_{\varepsilon,p} \times (\mu_{4d}(\Sigma)) \times  \| |(x,w)|^{\varepsilon} A(f,g) \|_{2 , 4d}^{\frac{4pd}{(2d+\varepsilon)}}(\|f\|_{2,2d}\|g\|_{2,2d})^{p(p-\frac{2d-\varepsilon}{2d+\varepsilon})}.$$
\begin{cor}
\mbox{}\\
Let $\varepsilon$ , p be two positive real numbers such that $0<\varepsilon<2d$ and $p \geq 1$, then there is a non negative constant $N_{\varepsilon,p}$ such that for every quaternion function $g$ , for every  function $f \in L^{2}(\mathbb{R}^{2d},\mathbb{H})$ and for every finite measurable subset $\Sigma$ of $\mathbb{R}^{2d}\times\mathbb{R}^{2d}$, we have
\begin{equation}\label{3}
\|\chi_{\Sigma}W(f,g)\|_{p,4d}^{p(p+1)} \leq N_{\varepsilon,p} \times (\mu_{4d}(\Sigma)) \times  \| |(x,w)|^{\varepsilon} W(f,g) \|_{2 , 4d}^{\frac{4pd}{(2d+\varepsilon)}}(\|f\|_{2,2d}\|g\|_{2,2d})^{p(p-\frac{2d-\varepsilon}{2d+\varepsilon})}
\end{equation}
where
$N_{\varepsilon,p} = 2^{2d(p-2)(p+1) + 4d + \frac{4pd\varepsilon}{(2d+\varepsilon)}}\bigg(\dfrac{2d+\varepsilon}{2^{\frac{\varepsilon(2d+2p+2)}{(2d+\varepsilon)(P+1)}}\varepsilon^{\frac{2\varepsilon}{2d+\varepsilon}}\Gamma(2d)^{\frac{\varepsilon}{(2d+\varepsilon)(p+1)}}(2d-\varepsilon)^{\frac{2d-\varepsilon}{2d+\varepsilon}+\frac{\varepsilon}{(2d+\varepsilon)(p+1)}}}\bigg)^{p(p+1)}$.
\end{cor}
\textbf{Proof}\\
We have
$$\begin{tabular}{lll}
$\|\chi_{\Sigma} W(f,g)\|_{p,4d}^{p}$ & $=$ & $2^{2dp} \displaystyle\iint_{\mathbb{R}^{2d}\times\mathbb{R}^{2d}} |\chi_{\Sigma}(x,w)\mathcal{G}_{\check{g}}(f)(2x,2w)|^{p} d\mu_{4d}(x,w)$\\
& $=$ & $2^{2d(p-2)} \displaystyle\iint_{\mathbb{R}^{2d}\times\mathbb{R}^{2d}} |\chi_{\Omega}(y,\lambda) \mathcal{G}_{\check{g}}(f)(y,\lambda)|^{p} d\mu_{4d}(y,\lambda)$\\
& $=$ & $2^{2d(p-2)} \|\chi_{\Omega} \mathcal{G}_{\check{g}}(f)\|_{p,4d}^{p}$
\end{tabular}$$
where
$$\Omega = \bigg\{ (y,\lambda)\in \mathbb{R}^{2d}\times\mathbb{R}^{2d} |(\dfrac{y}{2},\dfrac{\lambda}{2}) \in \Sigma \bigg\} = \{ (2x,2w) \in \mathbb{R}^{2d}\times\mathbb{R}^{2d} | (x,w) \in \Sigma \}$$

then from the relation of the Local Price's inequality  we get
$$\begin{tabular}{lll}
$\|\chi_{\Sigma} W(f,g) \|_{p,4d}^{p(p+1)}$& $=$ & $2^{2d(p-2)(p+1)} \|\chi_{\Omega} \mathcal{G}_{\check{g}}(f)\|_{p,4d}^{p(p+1)}$\\
& $ \leq$ & $ 2^{2d(p-2)(p+1)} M_{\varepsilon,p} \times (\mu_{4d}(\Omega)) \times  \| |(x,w)|^{\varepsilon} \mathcal{G}_{\check{g}}(f) \|_{2 , 4d}^{\frac{4pd}{(2d+\varepsilon)}}(\|f\|_{2,2d}\|g\|_{2,2d})^{p(p-\frac{2d-\varepsilon}{2d+\varepsilon})}$\\
& $\leq$ & $2^{2d(p-2)(p+1)}2^{4d} M_{\varepsilon,p} \times (\mu_{4d}(\Sigma)) \times  \| |(x,w)|^{\varepsilon} \mathcal{G}_{\check{g}}(f) \|_{2 , 4d}^{\frac{4pd}{(2d+\varepsilon)}}(\|f\|_{2,2d}\|g\|_{2,2d})^{p(p-\frac{2d-\varepsilon}{2d+\varepsilon})}.$\\
\end{tabular}$$
But
$$\begin{tabular}{lll}
$\| |(x,w)|^{\varepsilon} \mathcal{G}_{\check{g}}(f)\|_{2,4d}^{2}$ & $=$ & $\displaystyle\iint _{\mathbb{R}^{2d}\times\mathbb{R}^{2d}} |(x,w)|^{2\varepsilon} |\mathcal{G}_{\check{g}}(f)(x,w)|^{2} d\mu_{4d}(x,w)$\\
& $=$ & $ 2^{4d}\displaystyle\iint _{\mathbb{R}^{2d}\times\mathbb{R}^{2d}} |(2y,2\lambda)|^{2\varepsilon} |\mathcal{G}_{\check{g}}(f)(2y,2\lambda)|^{2} d\mu_{4d}(y,\lambda)$\\
& $=$ & $ \displaystyle\iint _{\mathbb{R}^{2d}\times\mathbb{R}^{2d}} |(2y,2\lambda)|^{2\varepsilon} |W(f,g)(y,\lambda)|^{2} d\mu_{4d}(y,\lambda)$\\
& $=$ & $2^{2\varepsilon} \displaystyle\iint _{\mathbb{R}^{2d}\times\mathbb{R}^{2d}} |(y,\lambda)|^{2\varepsilon} |W(f,g)(y,\lambda)|^{2} d\mu_{4d}(y,\lambda)$\\
& $=$ & $2^{2\varepsilon} \| | (x,w)|^{\varepsilon} W(f,g)\|_{2,4d}^{2}.$
\end{tabular}$$
We finally get
$$\begin{tabular}{lll}
$\|\chi_{\Sigma} W(f,g)\|_{p,4d}^{p(p+1)}$ & $\leq$ & $2^{2d(p-2)(p+1)}2^{4d} 2^{\frac{4pd\varepsilon}{(2d+\varepsilon)}} M_{\varepsilon,p} \times (\mu_{4d}(\Sigma)) \times  \| |(x,w)|^{\varepsilon} W(f,g) \|_{2 , 4d}^{\frac{4pd}{(2d+\varepsilon)}}$\\
& $\times$ & $(\|f\|_{2,2d}\|g\|_{2,2d})^{p(p-\frac{2d-\varepsilon}{2d+\varepsilon})}$\\
& $=$ & $ N_{\varepsilon,p} \mu_{4d}(\Sigma) \| |(x,w)|^{\varepsilon} W(f,g) \|_{2,4}^{\frac{4pd}{2d+\varepsilon}} (\|f\|_{2,2d} \|g\|_{2,2d})^{p(p-\frac{2d-\varepsilon}{2d+\varepsilon})}.$
\end{tabular}$$

\bibliographystyle{plain}

\end{document}